\documentclass[12pt,psfig]{article}
\usepackage{amsthm,amssymb,amsmath,amsfonts,amscd}
\usepackage{graphicx,epsfig,latexsym}
\usepackage{cite,calc,xcolor,subfigmat,mathrsfs,dsfont,epstopdf}
\usepackage{hyperref}
\usepackage{cleveref}
\usepackage{chngcntr}
\usepackage{tabularx}
\usepackage{algorithm}
\usepackage{algorithmic}
    \usepackage{pgfplots}
    \usepackage{tikz}
    \usepackage{circuitikz}
     \usepackage{listings}
     \usepackage{mathtools}
    \usepackage{caption}
    \usepackage{cite,calc,xcolor,mathrsfs,dsfont}
    \usepackage{xcolor}
    \usepackage{harmony}
    \usepackage[toc,page]{appendix}

    \newcommand{\B}{\mathscr{B}}
    
    \newcommand{\LST}{\mathscr{L}}
    \newcommand{\RScr}{\mathscr{R}}

    \newcommand{\V}{\mathscr{V}}

    \newcommand{\N}{\mathbb{N}}

    \newcommand{\dd}{\mathfrak{u}}
    
    \newcommand{\NZ}{\mathbb{N}_0}
    \newcommand{\Span}{{\rm span}}
    
    \newcommand{\rad}{{\rm rad}}

    \newcommand{\im}{{\rm im}}
    \newcommand{\ad}{{\rm ad}}
    \newcommand{\rank}{{\rm rank}}
    
    \newcommand{\0}{\mathbf{0}}
    \newcommand{\x}{\mathbf{x}}
    \newcommand{\z}{\mathbf{z}}

    \newcommand{\m}{\mathbf{m}}
    
    \newcommand{\ie}{{\em i.e.,} }
    \newcommand{\eg}{{\em e.g.,} }
    
    \newcommand{\IM}{{\rm im }}

    \newtheorem{thm}{Theorem}[section]
    \newtheorem{cor}[thm]{Corollary}
    \newtheorem{lem}[thm]{Lemma}
    \newtheorem{prop}[thm]{Proposition}
    \theoremstyle{definition}
    
    \newtheorem{exm}[thm]{Example}
    \newtheorem{rem}[thm]{Remark}

    \newcommand*{\Scale}[2][4]{\scalebox{#1}{$#2$}}%
    \numberwithin{equation}{section}
    \def\l {\label}
    
    \def\Blem {\begin{lem}}
    \def\Elem {\end{lem}}
    \def\be {\begin{equation}}
    \def\ee {\end{equation}}
    \def\ba {\begin{eqnarray}}
    \def\ea {\end{eqnarray}}
    \def\bes {\begin{equation*}}
    \def\ees {\end{equation*}}
    \def\bas {\begin{eqnarray*}}
    \def\eas {\end{eqnarray*}}
    \def\bpr {\begin{proof}}
    \def\epr {\end{proof}}
   \def\BState{\State\hskip-\ALG@thistlm}
    
\begin{document}
    \baselineskip=18pt
    \renewcommand {\thefootnote}{ }

    \pagestyle{empty}

\begin{center}
\leftline{}
\vspace{-0.500 in}
{\Large \bf Parametric normal form classification for Eulerian and rotational non-resonant double Hopf singularities
%\\Non-resonant double Hopf normal forms: root system analysis, Eulerian-rotational classifications and a robotic arm control
} \\ [0.3in]

{\large Majid Gazor\(^{\dag}\)\footnote{$^\dag\,$Corresponding author. Phone: (98-31) 33913634; Fax: (98-31) 33912602;
Email: mgazor@cc.iut.ac.ir; Email: ahmad.shoghi@math.iut.ac.ir.} and Ahmad Shoghi }

\vspace{0.105in} {\small {\em Department of Mathematical Sciences, Isfahan University of Technology
\\[-0.5ex] Isfahan 84156-83111, Iran }}

%\today

\vspace{0.05in}

\end{center}

\vspace{-0.10in}

    \baselineskip=16pt

%\:\:\:\:\ \ \rule{5.88in}{0.012in}

\begin{abstract}
In this paper we provide novel results on the infinite level normal form and orbital normal form classifications of nonlinear Eulerian and rotational vector fields with two pairs of non-resonant imaginary modes. We use the method of multiple Lie brackets and its extension along with time rescaling for orbital normal form classification. Furthermore, we apply two reduction techniques. The first is to use the radical Lie ideal of rotational vector fields and its corresponding quotient Lie algebra.  The second technique is to employ a Schur complement block matrix type in Gaussian elimination and analysis of block matrices. The infinite level parametric normal form classification are also presented. The latter is also viewed as a normal form result for multiple-input controlled systems with non-resonant double Hopf singularity. We also discuss {\it nonlinear symmetry transformations} associated with the {\it nonlinear symmetry group} of the simplest normal forms. Symbolic normal form transformation generators are derived for computer algebra implementation. Further, the results are efficiently implemented and verified using {\sc Maple} for all three types of normal form computations up to arbitrary degree, where they can also include both small bifurcation parameters and arbitrary symbolic constant coefficients.
\end{abstract}

\vspace{0.10in}


\vspace{0.10in} \noindent {\it 2010 Mathematics Subject Classification}:\, 34C20; 34A34; 34C14.

%\noindent \rule{5.88in}{0.012in}

\vspace{-0.1in}

\section{Introduction }

In this paper we are concerned with normal forms of Eulerian and rotational vector fields with a non-resonant double Hopf singularity.
Hence, we consider
\be\label{Eq1}
v(\x):= v_0+ E_{f}+ \Theta^1_{g_1}+ \Theta^2_{g_2},
\ee where \(v_0:=w_1\Theta^1_{0,0}+ w_2\Theta^2_{0,0},\) \(\omega_1\omega_2\neq0,\) \(\frac{\omega_1}{\omega_2}\notin\mathbb{Q},\) \(g_1(0)= g_2(0)= f(0)=0,\)
\bas
&E_f:= fE_{0,0},\quad E_{0,0}:= E^1_{0,0}+E^2_{0,0}, \quad E^i_{0,0}:= x_i\frac{\partial}{\partial x_i}+y_i\frac{\partial}{\partial y_i}, &\\& \Theta^i_{g_i}:= g_i(\x) \Theta^i_{0,0},\quad \Theta_{0, 0}^i:= -y_i\frac{\partial}{\partial x_i}+x_i\frac{\partial}{\partial y_i}, &
\eas for any \(g_1, g_2, f\in\mathbb{R}[[\x]],\) \(\x:=(x_1, y_1, x_2, y_2),\) \(g\neq0,\)  and \(i= 1, 2.\) We refer to \(E_{f} \) as an Eulerian, \(E^i_{f}= fE^i_{0, 0}\) as a radial vector field while \(\Theta^i_{g_i} \) stands for a {\it rotating vector field}. The vector field \eqref{Eq1} is associated with the differential system \(\frac{d\,x_i}{d\,t}= -w_iy_i-y_if_i+x_ig,\) \(\frac{d\,y_i}{d\,t}= w_ix_i+x_if_i+y_ig\) for \(i=1, 2.\)
%The vector field \eqref{Eq1} is called by a {\it plant} when it is to be controlled through a local state-feedback controller design.
We refer to
\be\label{InpSys} w:= v_0+ E_{F}+\Theta^1_{G_1}+ \Theta^2_{G_2}
\ee as a multiple-parametric perturbation of \(v(x),\) a parametric vector field
or a multiple input-system %or a state-feedback controlled-system
when \(F, G_i\in \mathbb{F}[[\x, \mu]]\), \(\mu\in \mathbb{R}^N,\) \(F(x, 0)= f(x),\) \(G_i(x, 0)= g_i(x),\) \(i=1, 2.\)

\pagestyle{myheadings} \markright{{\footnotesize {\it M. Gazor and A. Shoghi \hspace{1.45in} {\it Normal forms of Eulerian double-Hopf singularities}}}}

%The idea is to design the feedback controllers based on {\it distinguished parameters}, where they can play the role of the remaining parameters (universal asymptotic unfolding parameters \cite{GazorSadri}) in multi-input controlled normal forms. Then, a dynamics analysis of the controlled normal form in terms of controller parameters gives rise to the desired choices for the dynamics control.

The conventional approach is to exclude the input parameters of an input system  by setting them to zero and then, obtain the normal form of the system without inputs. Then, one derives a parametric model (called universal unfolding) by adding parametric terms to the normal form system so that the dynamics of the universal unfolding would represent the local dynamics of any possible small perturbations of the normal form system. Next, the bifurcation analysis of the universal unfolding concludes any possible bifurcation scenarios of the original input system. However, this does not describe the actual {\it quantitative} dynamical experience in terms of the original parameters. Hence, the only useful normal form for the actual bifurcation analysis and control of a real life problem is the controlled and parametric normal forms. These have recently been obtained for only a few cases; see \cite{GelfreichNonl2014,GelfreichNonl2009,GazorSadriBog,GazorMoazeni,GazorSadri,GazorYuSpec,YuLeung03} while we here treat the families \eqref{Eq1} and \eqref{InpSys}. The controlled and parametric normal forms are derived in a way to play the role universal unfolding of the original plant. Thus, the controller designs based on these will be robust against small unavoidable errors and perturbations; see \cite{GazorSadriBog,GazorSadri}. Furthermore, the truncated classical normal forms may destroy the Eulerian structural symmetry of the vector field \eqref{Eq1}. Hence, the truncated normal form system may represent a {\it qualitative} dynamics {\it inconsistent} with the original Eulerian dynamics. Therefore, the second goal here is to classify normal forms of the plant \eqref{Eq1} and controlled system \eqref{InpSys} so that their Eulerian structure are preserved in all normalization steps. This is possible when the set of normalizing transformations preserve the structural symmetry. These facilitate our third objective for a symbolic normal form computer algebra implementation.

In the last two decades, there have been numerous contributions on hypernormalization and classification of two dimensional state space systems; \eg see \cite{GelfreichNonl2014,GelfreichNonl2009,KokubuWang,ChenWang,chendora,wlhj,GazorYuSpec,GazorMoazeni,Stroyzyna17,ZoladekNF2015,Zoladek03} and the references therein. As the state dimension of the singularities increase, the complexity of hyper-normal form classifications significantly amplifies.
For the three-dimensional cases, there are only a few results for Hopf-zero and triple zero singularities; see \cite{Algaba30,Wang3DJDE2014,Wang3DIJBC2017a,GazorMokhtari3D,YuTriple,GazorSadri,GazorMokhtari,WangChenHopfZ,AlgabaHopfZ,YuHopfZero}. Contributions on normal form classification of three dimensional singularities use specific structural symmetries. They use and preserve the structure in their normal form results and/or use it for a normal form decomposition. However, there does not yet exist results on normal form classification with regards to non-resonant double Hopf singularity.

Sections \ref{Sec3}, \ref{Sec4}, and \ref{Sec5} treat normal forms of all generic and degenerate cases of vector field types \eqref{Eq1} by preserving their structural symmetry. However, in order to simplify the following formulas presented in the introduction, we assume that
\be\label{Generic}
{b_{0,1}}^2 b_{2,0}-b_{0,1}b_{1,0}b_{1,1}+b_{0,2}{b_{1,0}}^2\neq 0,
\ee and \(b_{1, 0}\neq 0\) when \(b_{i, j}\)-s are the first level normal form coefficients.
Then, the infinite level normal form of \eqref{Eq1} reads
\bas
&\dot{z_i}=\sum_{j+k=0}^{2}c^i_{j, k}|z_1|^{2j}|z_2|^{2k} z_i+\sum_{k\geq 3}c^i_{0, k}|z_2|^{2k} z_i, \quad \dot{w_i}= \dot{\overline{z_i}}= \overline{\dot{z_i}}, \; w_i= \overline{z_i}, \; i=1, 2, &
\eas where \(c^i_{j, k}:=b_{j,k}+Ia^i_{j,k},\) \(b_{0, 3}=a^i_{2,0}=a^i_{1, 1}=0, b_{0, 0}=0, a^i_{0, 0}=\omega_i,\) \(I^2=-1,\) \((z_i, w_i)\)  denotes the complex coordinates and the over-line stands for the complex conjugate. The infinite level coefficients are uniquely determined by equation \eqref{Eq1}. When $b_{1, 0}:=\frac{\partial^2 f}{4\partial x_1^2}(\0)+\frac{\partial^2 f}{4\partial y_1^2}(\0)\neq 0,$ the infinite level orbital normal form is
\bas
&\dot{z_i}=I\omega_i z_i+\sum_{j=0}^1c^i_{1-j, j}|z_1|^{2-2j}|z_2|^{2j}z_i+b_{0,2}|z_2|^{4}z_i+\sum_{j\geq 2}a^i_{0,j}I|z_2|^{2j}z_i, \quad \dot{w_i}= \overline{\dot{z_i}}, i=1, 2,&
\eas where \(a^1_{0, j}=0\) for \(j\geq 2\). Since \(b_{0, 2}\neq 0\), the input vector field \eqref{InpSys} can be uniquely transformed into
\bas
&\Scale[0.97]{\dot{z_1}=b_{0,2}|z_2|^{4}z_1
\!+\!\sum_{n=0}^{1}\sum_{j=0}^{n}\left(b_{n-j,j}(\mu)\!+\!a^1_{n-j,j}(\mu)I\right)|z_1|^{2(n-j)}|z_2|^{2j}z_1,}\qquad\qquad\qquad\qquad\quad&\\
&\Scale[0.965]{\dot{z_2}=\sum_{n=0}^{1}\sum_{j=0}^{n}\left(b_{n-j,j}(\mu)+a^2_{n-j,j}(\mu)I\right)|z_1|^{2(n-j)}|z_2|^{2j}z_2+b_{0,2}|z_2|^{4}z_2\!+\!\sum_{j=2}^\infty a^2_{0,j}(\mu)I|z_2|^{2j}z_2}&
\eas and \(\dot{w_i}= \overline{\dot{z_i}}\) for \(i=1, 2.\) Here, \(b_{1, 0}(\mu)=b_{1, 0},\) \(b_{0, 0}(\0)=0,\) \(a^2_{0, 0}(\0)=\omega_2\) and \(a^1_{0, 0}(\mu)=\omega_1.\)

The rest of this paper is organized as follows. Complete normal form classification of vector fields \eqref{Eq1} are derived in Section \ref{Sec3}; where we only use the changes of state variables. Proofs include deriving the transformation generator formulas for practical implementations in a computer algebra system. Near-identity time rescaling are also used for the orbital normal form classification in Section \ref{Sec4}. Section \ref{Sec5} treats the parametric normal form classification for multiple-input vector fields \eqref{InpSys}. An efficient algorithm is proposed in Section \ref{Sec6} for the normal form computation using computer algebra systems. Some normal form formulas are also derived for practical applications.
%We consider a controlled two-joint robotic arm in Section \ref{Examples}. The normal forms of the controlled robot arm are presented. Numerical simulation illustrates a successful feedback stabilization and oscillation designs for the robotic arm oscillator.

\section{ Infinite level normal forms }\label{Sec3}

For \(l, k\in \mathbb{Z},\) and \(i=1, 2,\) let
\begin{eqnarray}\label{ETheta}
&\Scale[0.94]{E_{k-l, l}=\left({x_1}^2+{y_1}^2 \right)^{k-l}\left({x_2}^2+{y_2}^2 \right)^{l}E_{0, 0}\quad
\text{ and }\quad \Theta_{k-l, l}^i :=\left({x_1}^2\!+\!{y_1}^2 \right)^{k-l}\left({x_2}^2\!+\!{y_2}^2 \right)^{l}\Theta_{0, 0}^i}.&
\end{eqnarray}
Define \(\LST= \mathbb{R}v_0+\{\sum^{\infty, k}_{k=0, l=0} a_{l, k}E_{k-l, l}+\sum^{\infty, k, 1}_{k=1, l=0, i=0} b^i_{l, k}\Theta_{k-l, l}^i\,|\, a_{l, k}, b^i_{l, k}\in \mathbb{R} \}.\) Any non-resonant double Hopf singularity can be transformed into a first level normal form given by a vector field \(v^{(1)}\in \LST.\) Hence, we call \(\LST\) as the space of all (first level) normal form vector fields. Assume that \(S\in \LST\) has no linear term in its power series expansion. Then, \(S\) generates a near-identity changes of state variables, that is, the time-one map flow associated with \(S.\) Therefore, we call \(S\) a transformation generator. A Lie bracket is defined by \([u, w] := uw- wu,\) where \(v\) and \(w\) are considered as differential operators; see \cite{KokubuWang,Wang3DIJBC2017a,MurdBook}. This provides a natural Lie algebra structure for \(\LST.\)
\begin{prop}[Structure constants] \label{structure constant 1}  The space \(\LST\) is a Lie algebra and its structure constants are given by
\ba\label{Struc}
&\Scale[0.9]{\left[\Theta_{m,n}^i,\Theta_{k,l}^j\right] =0, \,\,\left[\Theta_{m,n}^i, E_{k,l}\right]=2(m\!+\!n) \Theta_{m+k,n+l}^i,\,\,
\left[E_{m,n}, E_{k,l}\right] =2(m\!+\!n\!-\!k\!-\!l)E_{m+k,n+l}.}&
\ea
\end{prop}
\bpr The proof is a straightforward computation.
\epr
The normal form formulations here are presented using the method of multiple Lie brackets and matrix representations; \eg see \cite{wlhj,ChenWang,KokubuWang,chendora,WangChenHopfZ,Wang3DJDE2014,Wang3DIJBC2017a}. We provide recursive relations for the normal form transformation generators transforming the {\it updating} vector field into a higher level normal form. Following \cite{GazorSadriBog,GazorMoazeni,GazorYuSpec}, we {\it simultaneously} recall the theory of our infinite level hypernormalization steps for this section and the next three sections. Let \(\mathds{B}\) be either \(\LST\) or its parametric extension, \(\mathds{B}=\sum^\infty_{k=0} \mathds{B}_k\) be a \(\mathbb{Z}_{\geq0}\)-Lie graded structure for \(\mathds{B},\) \ie \([\mathds{B}_k,  \mathds{B}_l]\subseteq \mathds{B}_{l+k}\) for all \(l, k\in \mathbb{Z}_{\geq0},\) and for \(v_k\in \mathds{B}_k,\) \(v=\sum^\infty_{k=0} v_k\in \mathds{B}\) be an updating (\ie being normalized) non-resonant double Hopf singular vector field. Denote the graded linear space \(\mathds{A}=\sum^\infty_{k=1} \mathds{A}_k\) for the normalizing transformation generators and \(*\) for its graded action on \(\mathds{B}\), \ie \(\mathds{A}_k*\mathds{B}_l\subseteq \mathds{B}_{k+l}\) for all \(l\geq 0\) and \(k\geq 1\). The space \(\mathds{A}\) and action \(*\) are different in Sections \ref{Sec3}, \ref{Sec4}, and \ref{Sec5} and they will be defined in these sections accordingly. Define
\ba\label{dkr}
&d^{k, 1}: \mathds{A}_k\rightarrow \mathds{B}_k, \qquad \hbox{ by }\quad d^{k, 1}(X_k):= X_k*v_0 \quad \hbox{ for } \quad X_k\in \mathds{A}_k.&
\ea Assume that \(\RScr^{k, 1}:= \IM (d^{k, 1})\) and \(\mathcal{C}^{k, 1}\) denotes for its complement space that is uniquely determined via a normal form style. Then, \(\RScr^{k, 1}\oplus \mathcal{C}^{k, 1}= \mathds{B}_k\) and by \cite[Lemma 4.2]{GazorYuSpec}, there exists a sequence of near identity transformations sending \(v\) into the first level extended partial (orbital or parametric depending on the space \(\mathds{A}\)) normal form \(v^{(1)}:= v^{(1)}_0+\sum^\infty_{k=1} v^{(1)}_k,\) where \(v^{(1)}_k\in\mathcal{C}^{k, 1}\) and \(v_0=v^{(1)}_0\); also see \cite{Murd09}. The idea is to use the transformations generated by \(\mathds{A}_k\) to eliminate all terms living in \(\RScr^{k, 1}\) from the normalizing vector field. Since the space \(\ker(d^{k, 1})\) does not contribute to the simplification of terms in grade \(k,\) we shall use them in normalizing higher graded terms. A systematic derivation of hypernormalization steps to infinity is required for derivation and computer algebra implementation of the infinite level normal forms, \ie no further normalization is possible. In each normalization step, one needs to simultaneously track the effects of the normalizing transformations to the normalizing vector field and also derive the available normalizing transformations for higher level hypernormalization steps. This is naturally reflected to the computational burden for the normal form classification as the state-dimension of the singularity increases. Thereby,
we inductively denote
\be\label{dkr1}
d^{k,r}: \ker(d^{k-1,r-1})\times \mathds{A}_k\rightarrow \mathds{B}_k \quad (\hbox{for any } r\leq k),
\ee given by \(d^{k,r}(X_{k-r+1}, \ldots, X_{k-1}, X_{k}):= \sum^{r-1}_{i=0} X_{k-i}*v^{(r-1)}_i,\) where
\be\label{TransOrder}
(X_{k-r+1}, X_{k-r+2}, \cdots, X_{k-1})\in \ker(d^{k-1,r-1})
\ee for the \(r\)-th level map; also see the differential of bi-degree \((r, 1-r)\) defined on \cite[page 1015]{GazorYuSpec}. For any \(r>k,\) let \(d^{k,r}:=d^{k,k}.\) Let \(\RScr^{k,r}:= \IM (d^{k,r})\) and \(\mathcal{C}^{k, r}\) be its complement subspace with respect to a formal basis style, \ie \(\RScr^{k,r}\oplus \mathcal{C}^{k,r}= \mathds{B}_k.\) The complement space associated with formal basis styles are generated by {\it Eulerian terms} and {\it rotational terms}; see \cite{GazorYuSpec} for more details on formal basis style. The space of all rotational vector fields constitutes the radical Lie ideal of Lie algebra \(\LST\) and provides a reduction technique for normal form computations; see \cite{GazorShoghiLie} for a proof of our claim. In particular, \bes
\mathcal{C}^{k,r}= \pi_{\rad\, \LST}(\mathcal{C}^{k,r})\oplus\Pi(\mathcal{C}^{k,r}+\rad\,\LST)\quad \hbox{ and }\quad \pi_{\rad\, \LST}(\mathcal{C}^{k,r})= \mathcal{C}^{k,r}\cap \rad\, \LST.
\ees Let \({\rm im }\,d^{k, r}+\rad\,\LST\) be a linear subspace of quotient Lie algebra \(\frac{\LST_k+\rad\, \LST}{\rad\,\LST}.\) Since the formal basis style gives the priority of elimination to Eulerian terms than rotational terms of the same grade, the complement space for \(\RScr^{k,r}\) is the same as the complement space for \(({\rm im }\, d^{k, r}\cap\, \rad\,\LST)+\Pi({\rm im }\, d^{k, r}+\rad\,\LST),\) \ie
\begin{itemize}
  \item The linear space \(\Pi({\rm im }\, d^{k, r}+\rad\,\LST)\) determines all the normalizing {\it Eulerian terms} in the \(r\)-th level.
  \item The normalizing {\it rotational terms} in the \(r\)-th level normalization step are determined by \({\rm im }\, d^{k, r}\cap\,\rad\,\LST.\)
\end{itemize}
These explain the proofs in the following sections, where we present recursive formulas for transformation generators and their impact on the normalizing vector field. Hence,
\ba\label{TwoEqs}
&\Scale[0.85]{({\rm im }\, d^{k, r}\cap{\rm rad}\LST)\oplus(\mathcal{C}^{k,r}\cap{\rm rad}\LST)\!=\! \rad\LST\cap \LST_k
\hbox{ and } ({\rm im }d^{k, r}\!+\!{\rm rad}\LST)\oplus (\mathcal{C}^{k,r}\!+\!{\rm rad}\LST)\!=\!\frac{\LST_k\!+\!{\rm rad}\LST}{{\rm rad}\LST}.}&
\ea Equations \eqref{TwoEqs} suggest two reductional techniques for the computation of complement spaces in \(\LST_k\):
\begin{enumerate}
  \item Possible restriction of the homological differential maps \(d^{k, r}\) on the {\it radical Lie ideal}.
  \item Introduction of a reduced map \(\hat{d}^{l+r, r+1}\) based on the factor algebra \(\frac{\LST}{\rad\LST}\); see equation \eqref{dhat}.
\end{enumerate}
When all \(k\)-grade-homogenous parts \(v_k\) of a vector field \(v\) belongs to \(\mathcal{C}^{k,r}= (\mathcal{C}^{k,r}\cap\rad\,\LST)\oplus\Pi(\mathcal{C}^{k,r}+\rad\,\LST)\), the vector field is called a \(r\)-th level extended partial (orbital or parametric) normal form. The vector field \(v\) is called the infinite level (orbital or parametric) normal form, when \(v_k\in \mathcal{C}^{k,k}\) for all natural numbers \(k\). The coefficients of the infinite level normal forms are uniquely determined by the original vector field.

\begin{thm}\cite[Theorem 4.4]{GazorYuSpec}\label{LNF} Consider a formal basis normal form style, a Lie graded structure for \(\mathds{B}\) and a grading-module structure for \(\mathds{B}\) over the transformation (generator) space \(\mathds{A}.\) Then for any vector field \(v\in \mathds{B},\) there is a sequence of near-identity transformations so that they transform \(v\) into its \(r\)-th level extended partial (orbital or parametric depending on the transformation space \(\mathds{A}\)) normal form \(v^{(r)}\) and infinite level normal form \(v^{(\infty)}\).
\end{thm}

Denote \(\mathbf{e}^i_{j}:=(0, 0, \cdots, 1, \cdots, 0)\in \mathbb{R}^j\) for the \(i\)-th element of the standard basis in \(\mathbb{R}^j.\) The index for a bold zero denotes the dimension of a zero vector, \ie \(\0_k\in \mathbb{R}^k\). We, however, skip the indices when it does not lead to a confusion. We use double, triple, etc, indices for summations (or linear subspace spans), when we deal with double or triple sums; \eg we denote \(\sum^{k, n}_{j=1, j=2k}a_j\) for \(\sum^{n}_{j=2k}a_j+\sum^{k}_{j=1}a_j.\) The rest of this section deals with normalization of vector fields \eqref{Eq1} by only using changes of state variables.
\begin{lem}\label{1stLevel}
There exists a sequence of near-identity changes of state variables that they transform any vector field given by \eqref{Eq1} into
\ba\label{Eq2}
&v^{(1)}=\sum^2_{i=1}\sum_{m+n=0}^{\infty}a^i_{m,n}\Theta_{m,n}^i+\sum_{m+n= 1}^{\infty}b_{m,n}E_{m,n}\in \LST,&
\ea where \(a^i_{m,n}, b^i_{m,n}\in \mathbb{R},\) \(m, n\in \NZ:=\mathbb{N}\cup\{0\},\) and \(a^i_{0,0}= \omega_i.\)
\end{lem}
\bpr Since the space of vector field types \eqref{Eq1} defined by
\bes
\V:=\Span\left\{ E_{f}, \Theta^i_{g_i}|\, f, g_i\in \mathbb{F}[[\x]], i=1, 2\right\}
\ees
is a Lie algebra, transformation generators from \(\V\) transform the vector field \eqref{Eq1} into a vector field in \(\V.\) Using formulas \eqref{Struc}, the linear part of the vector field \(v_0=\omega_1\Theta_{0, 0}^1+ \omega_2\Theta_{0, 0}^2,\) and the assumption \(\omega_1\omega_2\neq0,\) \(\frac{\omega_1}{\omega_2}\notin\mathbb{Q},\) the first level normal form vector field \(v^{(1)}\) holds a two-torus symmetry and has an invariant algebra generated by the two-torus invariants \({r_i}^2:= {x_i}^2+{y_i}^2\) for \(i=1, 2,\) \ie \(v^{(1)}\in \LST\).
%we have
%\bes[\omega_1\Theta_{0, 0}^1+ \omega_2\Theta_{0, 0}^2, E_{\psi_{m,n}}]={\color{blue} 0} \quad\hbox{ and } \quad [\omega_1\Theta_{0, 0}^1+ \omega_2\Theta_{0, 0}^2, \Theta^i_{\psi_{m,n}}]={\color{blue} 0}.
%\ees It is easy to see that all terms of the form ?? can be simplified. Then, the proof follows the equations \eqref{ETheta} and \eqref{LST}.
\epr

The Eulerian and rotating structure of the vector field types \eqref{Eq1} are preserved in further hypernormalization steps as long as the normalizing transformations are derived from \(\LST.\) Thus, normal form classifications in this paper deal with vector fields from \(\mathds{B}:=\LST\) with linear part \(v_0:=\omega_1\Theta_{0, 0}^1+ \omega_2\Theta_{0, 0}^2.\) The space of permissible transformation generators \(\mathds{A}\) is \([\LST, \LST]\). Let
\begin{equation}\label{sp}
s:=\min \left\lbrace m\geq 1 \vert\, \exists\, i\leq m,\ b_{m-i,i}\neq 0 \right\rbrace \quad \hbox{and }\quad p:=\min \lbrace i\,\vert\, b_{s-i, i} \neq 0\rbrace.
\end{equation} Then, \(s<\infty\) and \(p\leq s.\) Define a grading function $\delta$ by
\begin{equation*}
\delta(E_{m,n})=m+n, \quad \delta(\Theta_{m,n}^i)= s+m+n, \quad \hbox{ for }\; i=1, 2,\, \hbox{ and }\, m, n \in \mathbb{Z}_{\geq0}.
\end{equation*}

\begin{lem}\label{s+1level}
Let \(s<\infty.\) Then, the $s+1$-st level normal form of $v^{(1)}$ is given by
\begin{eqnarray*}
&v^{(s+1)}=v_0\!+\!\sum_{j=p}^{s} b_{s-j,j}E_{s-j,j}\!+\!\sum_{j=0}^{2s} b_{2s-j, j}E_{2s-j,j}\!+\!\sum_{l=1, i=1, j=0, j=l+p+1}^{\infty, 2, p-1, l+s} a^i_{l+s-j, j}\Theta_{l+s-j, j}^i&\\&+\sum^\infty_{l=1, {l\neq s}}\sum_{j=0,j=l+p+1 }^{p-1, l+s}b_{l+s-j, j}E_{l+s-j, j}+ \sum_{l+j= 1, i=1}^{s, 2}a^i_{l, j}\Theta_{l, j}^i.\qquad\;\;\,&
\end{eqnarray*}
\end{lem}
\begin{proof}
Let
\ba\label{Yslk}
&\Scale[0.95]{S^s_{l+k}\!:=\!\sum_{j=0}^{l+k}\!c_{l+k-j,j}E_{l+k-j,j}\!+\!\sum^2_{i=1}\!\sum_{j=0}^{l+k-s}d_{l+k-j-s,j}^i\Theta^i_{l+k-j-s,j}\in \LST_{l+k} \hbox{ for } 0\leq k<s,} &
\ea where \(d^i_{l+k-s-j}=0\) for \(l+k\leq s,\) denote the available transformation generator of grade \(l\) for the \(s+1\)-level hypernormalization step, \ie
\(\left(S^s_{l}, S^s_{l+1}, \cdots, S^s_{l+s-1}\right)\in \ker d^{l+s-1, s} \).
Hence,
\bas
&{\rm im}\, d^{l+s, s+1}= \Span^l_{j=0} \{d^{l+s, s+1}(E_{l-j,j},\0_s)\}\oplus \ad_{v_s}(\rad\,\LST\cap \LST_{l-s}).&
\eas Let \(v_{l+s}=\sum_{j=0}^{l+s}b_{l+s-j,j}E_{l+s-j,j}+\sum_{i=1, j=0}^{2, l}a_{l-j,j}^i\Theta^i_{l-j,j}\in \LST_{l+s}.\) Since \(b_{s-p,p}\neq0,\) we propose \(c_{0}:=\frac{b_{l+s-p,p}}{-2(l-s)b_{s-p,p}},\)
\begin{eqnarray}\label{RecurRelations1}
&c_{j}:= \frac{b_{l+s-p-j,p+j}-\sum_{i=1}^{j-1}c_{i} b_{s-p+i-j,p+j-i}}{-2(l-s)b_{s-p,p}}\hbox{ for }  1\leq j\leq s-p,&\\\nonumber &
\hbox{ and }\quad c_{j}:= \frac{b_{l+s-p-j,p+j}-\sum_{j=0}^{s-p-1}c_{i-s+p+j} b_{j,s-j}}{-2(l-s)b_{s-p,p}}&
\end{eqnarray} for \(s-p+1\leq j \leq l.\) These choices for \(c_{l-j, j}\) eliminate all \(b_{l+s-j,j}E_{l+s-j,j}\)-terms when \(j=p, \cdots, p+l\) and \(s\neq l\geq 1.\) However, \(b_{2s-j,j}E_{2s-j,j}\)-terms for all \(j\leq 2s\) may appear in the \(s+1\)-level normalization step. We remark that the choices in \eqref{RecurRelations1} does not exist when \(l\leq s-p.\) By restricting the differential map \(d^{l+2s, s+1}\) on the radical Lie ideal and a similar argument, all \(\Theta^i_{l-j, j}\)-terms for \(p\leq j\leq l+p\) and \(i=1, 2\) can be eliminated in the \(s+1\)-level normalization step when \(l\geq s+1\). Due to the rank condition \(\rank\,d^{l+s, s+1}=l+1\) when \(1\leq l\leq s-1,\) \(\rank\,d^{2s, s+1}=0,\) and \(\rank\,d^{l+s, s+1}=3l-2s+3\) if \(l\geq s+1,\) further normalization in the \((s+1)\)-th level is not possible.
\end{proof} Now let
\be\label{rq}
r:=\min \left\lbrace m>s \vert\, \exists\, j\leq m, b_{m-j,j}\neq 0 \right\rbrace,\ q:=\min \lbrace j \vert b_{r-j, j} \neq 0\rbrace, \; \hbox{and }\; q\leq r,
\ee
and update the grading function $\delta$ by
\begin{equation*}
\delta(E_{m,n})=m+n, \quad \delta(\Theta_{m,n}^i)= r+m+n, \quad \hbox{ for }\; i=1, 2,\, \hbox{ and }\, m, n \in \mathbb{Z}_{\geq0}.
\end{equation*} This update in the grading structure is compatible with our normal form algorithm. Now we assume that \(s<r<\infty\) and treat the cases for \((s<\infty, r=\infty)\) and \((r=s=\infty)\) in Theorem \ref{ThmRinfty}.

\begin{thm} \label{InftLevel} Assume that \(r, s<\infty\) for \(r, s\) in equations \eqref{sp} and \eqref{rq}. When \(q<p\), there exists \(\dd\in\N\cup\lbrace0\rbrace\) such that the $(r+1)$-th level normal form \(v^{(r+1)}\) of $v^{(1)}$ in equation \eqref{Eq1} is
given by
\begin{eqnarray*}
&v_0\!+\!\sum_{j=p}^{s}\! b_{s-i, i}E_{s-i,i}\!+\!\sum_{i+j=2s}b_{i,j}E_{i,j}+\sum_{l+j= 1, i=1}^{s, 2}\!a^i_{l, j}\Theta_{l, j}^i\!+\!\sum_{j=0,j=p+r+\dd_s+1}^{q-1, r+s}\!b_{r+s-j,j}E_{r+s-j,j}&\\&
+\sum_{l=1, i=1, j=0,j=l+p+1}^{\infty, 2, p-1,l+s}a^i_{l+s-j, j}\Theta_{l+s-j, j}^i
+\sum_{\substack{ l=k\\l\neq s, 2s-r}}^{\infty}\sum_{j=0, j=p+r-s+1}^{p-1,l+r}b_{l+r-j, j}E_{l+r-j, j}.\qquad\quad&
\end{eqnarray*} Here, \(k=0\) for \(r\neq 2s\) while \(k=1\) for \(r=2s.\) When \(p<q,\) we have \({\rm im}\, d^{l+r,r+1}= {\rm im}\, d^{l+r,s+1}\) for \(l\neq s,\) and \({\rm im}\, d^{r+s,r+1}= {\rm im}\, d^{r+s,s+1}+ \Span\{E_{s+r-j, j} |\, p+r+1\leq j\leq p+r+\dd_s\}.\)
\end{thm}
\begin{proof} Due to the properties of \(r\) and \(s,\)
\bas
&\Scale[0.9]{\ker d^{l+r-1,r}\!=\!\oplus_{j=p}^{s}\mathbb{R}E_{s-j,j} \mathbf{e}^{s-l+1}_{r}\!+\!\oplus_{j=0, k=0}^{l+r-s+k, s-1} \mathbb{R}E_{l+r-s+k-j, j}\mathbf{e}^{r-s+k+1}_{r}\!+\!\oplus^{l-1, k, 2}_{k=0, j=0, i=1} \mathbb{R}\Theta_{k-j,j}^i\mathbf{e}^{r-l+k+1}_{r},}&
\eas when $0\leq s-l\leq r-s-1,$ and otherwise
\bas
&\Scale[0.9]{\ker d^{l+r-1,r}=\Span^{s-1, l+r-s+k}_{k=0, j=0}\lbrace\left(\0_{r-s+k}, E_{l+r-s+k-j,j}, \0_{s-k-1}\right) \rbrace
+\oplus\,\Span_{i=1, k=0, j=0}^{2,s-1,l-s+k}\lbrace \left(\0, \Theta_{l-s+k-j,j}^i, \0\right)\rbrace.}&
\eas Hence, we only discuss $l=s$. Consider an ordered vector basis \(\B^{r+s, r+1}\) for \(\ker d^{r+s-1,r}\times\LST_{r+s}\) given by
\begin{eqnarray*}
&\B^{r+s, r+1}:= \cup_{j=p}^{s} \left\lbrace E_{s-j, j}\mathbf{e}^{1}_{r+1}\right\rbrace  \cup_{k=0, j=0, i=1}^{s, r+k, 2}\left\lbrace E_{r+k-j, j}\mathbf{e}^{k+r-s+1}_{r+1}, \Theta_{k-j,j}^i\mathbf{e}^{k+r-s+1}_{r+1}\right\rbrace, &
\end{eqnarray*} where the ordering \(\prec\) is uniquely determined by the following rules: (1) The basis terms of lower grades precede higher grade basis terms, (2) Rotational terms succeed Eulerian terms of the same grade, (3) \(\Theta^1\)-terms precede \(\Theta^2\)-terms of the same grade, (4)
\(E_{k-j, j}\mathbf{e}\prec E_{k-m, m}\mathbf{e}\) and \(\Theta^{i}_{k-j, j}\mathbf{e}\prec \Theta^{i}_{k-m, m}\mathbf{e}\) when \(j<m\) for \(i=1, 2,\) \(k\geq m\) and the corresponding standard basis vector \(\mathbf{e}.\) A ordered basis for \(\LST_{r+s} \) is given by
\begin{eqnarray*}
\B_{r+s}:=\lbrace E_{r+s,0}, E_{r+s-1,1}, \cdots , E_{0,r+s} \rbrace, \quad E_{r+s-i, i} \prec E_{r+s-j, j} \quad \hbox{ when } i\prec j.
\end{eqnarray*}
Then, the matrix representation of \(d^{r+s,r+1}\) with respect to \((\B^{r+s, r+1}, \prec)\) and \((\B_{r+s}, \prec)\) is given by
\be\label{drs}
\left[d^{r+s,r+1}\right]_{\B^{r+s, r+1}, \B_{r+s}}=
\begin{bmatrix}
2(s-r)\mathcal{M}_{r}^{s}   & 2(r-s)\mathcal{M}_{s}^{r}   & \0      \\
\0                          & \0                          & \0
\end{bmatrix}.
\ee The first two columns of block matrices in equation \eqref{drs} are associated with \(E\)-term transformation generators of grade \(r\) and \(s\), respectively. Similarly, the third and fourth columns of block-matrices correspond to \(\Theta\)-terms of grade \(r\) and \(s\). Now assume that \(q<p.\) Then,
\bes
\Pi({\rm im}\, d^{r+s,r+1}+\rad\,\LST)\cap \Span \{E_{r+s-j,j}\,|\, 0\leq j \leq q-1\}=\{\0\}.
\ees Since $b_{r-q,q}\neq 0,$ terms of the form $b_{r+s-j,j}E_{r+s-j,j}$ for $q\leq j\leq p-1 $ can be simplified through
\bas
&d^{r+s, r+1}\left(\sum_{j=0}^{p-q-1}c_{s-j,j}E_{s-j,j}, \0_{r+s}\right)= -\sum_{j=q}^{p-1} b_{r+s-j,j}E_{r+s-j,j}&
\eas
where \( c_{s, 0}:= \frac{b_{r+s-q,q}}{2(r-s)b_{r-q,q}},\) and $c_{s-j,j}$ for $1\leq j \leq p-q-1$ follows the recursive equations
\begin{eqnarray}\label{recursive relations}
&c_{s-j,j}:= \frac{b_{r+s-q-j,j}-\sum_{i=0}^{j-1}c_{s-i, i} b_{r-q-j+i,q+j-i}}{2(r-s)b_{r-q,q}}.&
\end{eqnarray} Now omit all zero-block sub-matrices along with the first $p$ rows and the first $p-q$ columns of matrix representation $d^{r+s,r+1}$ to obtain
\ba\label{SubMatrices}
&\begin{bmatrix}
A   &   B\\
C   &   D
\end{bmatrix}&
\ea where \(A, B, C, D\) are matrices of sizes \((r+1)\times (s-p+q+1),\) \( (r+1)\times(r+1),\) \((s-p)\times(s-p+q+1),\) and \((s-p)\times(r+1),\) respectively. The matrix \(B\) is a lower triangular matrix where the entries on the main diagonal are constant and equal to \(b_{s-p,p}.\) Since \(b_{s-p,p}\neq 0,\) \(B\) is invertible and
\begin{equation*}
\begin{bmatrix}
B^{-1}    &  0_{(r+1)\times(s-p)}\\
-D B^{-1} &  I_{(s-p)\times(s-p)}
\end{bmatrix}
\begin{bmatrix}
A  &  B\\
C  &  D
\end{bmatrix}=
\begin{bmatrix}
B^{-1}A       &    I_{(r+1)\times (r+1)}\\
-D B^{-1}A+C  &    \0_{(s-p)\times(r+1)}
\end{bmatrix}.
\end{equation*}
Let
\be\label{drankD} \dd_s:= \rank (C-D B^{-1}A).\ee The matrix \(C-D B^{-1}A\) plays a similar role to the Schur complement of a block in Gaussian elimination of a block matrix. The index \(s\) stands for consistency with \(\dd_l\) in equation \eqref{ul1}. Given the matrix representation for \(d^{r+s, r+1}\), \(d+r+1\)-terms associated with the first rows of the matrix \eqref{SubMatrices} are simplified, \ie \bes\{E_{r+s-j,j}\,|\, p \leq j \leq p+r+\dd_s\}\subseteq \Pi({\rm im}\, d^{r+s,r+1}+\rad\,\LST).\ees
Thus, we can simplify \(E_{r+s-j, j}\)-terms for \( r+p+1\leq j\leq r+p+\dd_s\) in the \(r+1\)-level. Since there are \(s-p\)-rows in the matrix \(C,\) we have \(\dd_s\leq s-p.\) For \(\dd_s<s-p,\) all terms \(E_{r+s-j, j}\) for \(r+p+\dd_s<j\leq r+s\) can still appear the \(r+1\)-level normal form. When \(\dd_s=s-p,\) \(E_{r+s-j,j}\)-terms for $p\leq j\leq r+s$ do not appear in the \((r+1)\)-level hypernormalization step.

Let \(q\geq p.\) The case \(q=p\) occurs only when \(r=2s\). Since \(b_{s-p, p}\neq 0\), for \(p\leq j\leq q-1\) we can directly simplify \(E_{r+s-j,j}\)-terms in the \(r+1\)-level using the first \(q\) rows of the representation matrix. However, note that the first \(p\)-rows are zero row vectors. Hence, we eliminate the first \(q\) rows of block matrix \(\left[2(s-r) \mathcal{M}_{r}^{s}\quad 2(r-s)\mathcal{M}_{s}^{r} \right]\) and the first \(q-p\) columns of sub-block matrix \(2(r-s)\mathcal{M}_{s}^{r}\) in the \(r+1\)-level map. Then, we obtain a blocked matrix of type \eqref{SubMatrices} where \(A, B, C, D\) are matrices of sizes \( (r-q+p+1)\times(s+1),\) \((r-q+p+1)\times(r-q+p+1),\) \((s-p)\times(s+1),\) and \((s-p)\times (r-q+p+1).\) The matrix \(B\) is a lower triangular matrix and the entries on its diagonal are the constant \(b_{s-p,p}.\) Therefore, \(B\) is invertible and
\begin{equation*}
\begin{bmatrix}
B^{-1}      &    0_{(r-q+p+1)\times(s-p)}\\
-D B^{-1}   &    I_{(s-p)\times(s-p)}
\end{bmatrix}
\begin{bmatrix}
A       &       B\\
C       &       D
\end{bmatrix}=
\begin{bmatrix}
B^{-1}A      &   I_{(r-q+p+1)\times(r-q+p+1)} &  \\
-DB^{-1}A+C & 0_{(s-p)\times(r-q+p+1)} &
\end{bmatrix}.
\end{equation*} Given \(\dd_s\) defined by equation \eqref{drankD}, terms of the form \(E_{r+s-j, j}\) for \(p\leq j\leq r+p+d\) are simplified while \(E_{r+s-j, j}\)-terms for \(0\leq j\leq p-1\) and \(r+p+d+1\leq j\leq r+s\) may appear in the \(r+1\)-th level normal form. When \(\dd_s=s-p,\) all \(E_{s+r-j,j}\)-terms for $p\leq j\leq s+r$ are simplified in the \((r+1)\)-level.
\end{proof}

    %%%%%%%%%%%%%%%%%%%%%%%%%%%%%%%%%%%%%%%%%%%%%%%%%%%%%%%%

\begin{exm}
The map \(d^{r+s,r+1}\) does not always have a full rank. For instance, let
\begin{eqnarray*}
v^{(3)}:= v_0 + a_{2, 0} E_{2, 0}+ a_{1, 1} E_{1, 1}+ a_{0, 2} E_{0, 2}+ a_{1, 2} E_{1, 2}+ a_{0, 3}E_{0, 3}+ \cdots,
\end{eqnarray*} where \(s=2,\) \(r=3,\) \(p=0,\) \(a_{3, 0}\) and \(a_{2, 1}\) are simplified in the third level normalization step. Thus, we have \(a_{2, 0}\neq0\) and \((a_{1, 2}, a_{3, 0})\neq (0, 0).\) Let \(a_{2, 0}= a_{1, 1}= 4a_{0, 2}=a_{0, 3}=\frac{a_{1, 2}}{2}= 1,\) and \(q=2.\) Thereby after removing the zero blocks, \(\left[d^{5 ,4}\right]_{\B^{5, 4}, \B_{5}}\) is given by
\begin{equation*}
\begin{bmatrix}
a_{2, 0} & 0          & 0        & 0         & 0        & 0        & 0 \\
a_{1, 1} & a_{2, 0}   & 0        & 0         & 0        & 0        & 0 \\
a_{0, 2} & a_{1, 1}   & a_{2, 0} & 0         & a_{1, 2} & 0        & 0      \\
0        & a_{0, 2}   & a_{1, 1} & a_{2, 0}  & a_{0, 3} & a_{1, 2} & 0 \\
0        & 0          & a_{0, 2} & a_{1, 1}  & 0        & a_{0, 3} & a_{1, 2} \\
0        & 0          &   0      & a_{0, 2}  & 0        & 0        & a_{0, 3}
\end{bmatrix},\quad \begin{bmatrix}
A       &       B\\
C       &       D
\end{bmatrix}=
\begin{bmatrix}
a_{2, 0} & 0       &|  & a_{1, 2} & 0        & 0      \\
a_{1, 1} & a_{2, 0}&|  & a_{0, 3} & a_{1, 2} & 0 \\
a_{0, 2} & a_{1, 1}&|  & 0        & a_{0, 3} & a_{1, 2} \\
-        &-        &|  & -  & -   & -\\
0      & a_{0, 2}  &|  & 0        & 0        & a_{0, 3}
\end{bmatrix},
\end{equation*} \(C-DB^{-1}A=[0\; 0]\), \(\dd_2=0,\) and \(\rank\left[d^{5 ,4}\right]_{\B^{5, 4}, \B_{5}}= 5\). When \(a_{0, 3}\neq0\) and \(a_{1, 2}=0,\) \(A= [a_{2, 0}\; a_{1, 1}\; a_{0, 2}]^T,\) \(B= a_{0, 3}I_{3\times 3}\) (a three by three diagonal matrix), \(q=3\) and \(\rank\left[d^{5 ,4}\right]_{\B^{5, 4}, \B_{5}}=6\).
\end{exm}

\begin{prop}\label{3.7Rem} Assume that the rank condition \eqref{drankD} and the hypothesis of Theorem \ref{InftLevel} hold. When either \(p<q\) or \(r\neq 2s\) holds, \(r+2\leq \rank\, d^{r+s,r+1}\leq r+s+1\) while \(2s+1\leq \rank\, d^{3s,2s+1}\leq 3s+1.\) Furthermore, \(q>p+r-s\) when \(r\neq 2s\) and \(q\geq p.\)
\end{prop}

%%%%%%%%%%%%%%%%%%%%%%%%%%%%%%%%%%%%%%%%%%%%%%%%%%%%%%%

\begin{thm}\label{InfNFCor} Consider equations \eqref{sp} and \eqref{rq}. Assume that \(\rank\left[d^{r+s,r+1}\right]_{\B_s\cup \B_r, \B_{r+s}}=r+s+1\). Then, \((r+1)\)-th level normal form of \(v^{(1)}\) is given by
\begin{eqnarray*}
v^{(r+1)}=&v_0+\sum_{j=p}^{s} b_{s-i, i}E_{s-i,i} +\sum_{i+j=2s}b_{i,j}E_{i,j}+\sum_{l+j= 1, i=1}^{s, 2}a^i_{l, j}\Theta_{l, j}^i
\qquad\qquad\qquad\qquad\quad&\\&+\sum_{l=1, i=1, j=-l, j=p+1}^{\infty, 2, p-l-1, s}a^i_{s-j, j+l}\Theta_{s-j, j+l}^i+\sum_{\substack{ l=0\\l\neq s, 2s-r}}^{\infty}\sum_{j=-l, j=p+r-s+1}^{p-l-1, r}b_{r-j, j+l}E_{r-j, j+l}.&
\end{eqnarray*} Furthermore, \(v^{(r+1)}\) constitutes the infinite level normal form of \(v^{(1)}.\) There does not exist any nontrivial nonlinear symmetry transformation generator within \(\LST\) associated with \(v^{(\infty)}=v^{(r+1)}.\)
\end{thm}
\begin{proof} By Proposition \ref{3.7Rem}, we obtain an invertible matrix by removing the \(s+1\)-th column of the matrix representation \(\left[d^{r+s,r+1}\right]_{\B_s\cup \B_r, \B_{r+s}}\). Hence, we remove the last column of \(\mathcal{M}_{r}^{s}\) and denote it by \(\hat{\mathcal{M}_{r}^{s}}.\) Then, we introduce the transformation coefficients by
\begin{eqnarray*}
&(c_{s, 0}, c_{s-1, 1}, \cdots, c_{1, s-1}, c_{r, 0},\cdots, c_{0, r})^t:= \frac{1}{2(r-s)}\left[ -\hat{\mathcal{M}_{r}^{s}}\quad \mathcal{M}_{s}^{r} \right]^{-1}(b_{r+s, 0},\cdots, b_{0, r+s})^t.&
\end{eqnarray*} Therefore,
\bas
&d^{r+s, r+1}\left(\sum_{j=0}^{s}c_{s-j,j}E_{s-j,j},\0, \sum_{j=0}^{r}c_{r-j,j}E_{r-j,j}, \0\right)= -\sum_{j=0}^{r+s} b_{r+s-j,j}E_{r+s-j,j},&
\eas and all \(E_{r+s-j,j}\)-terms for $0\leq j\leq r+s$ are simplified in the \((r+1)\)-level hyper normalization step.

Let \(v^{(r+1)}= v_0+ v_s+ v_r+ v_{2s}+ h.o.t.,\) \(v_i\in \LST_i.\) We show that for \(l\geq 1,\)
\begin{equation}\label{infinite relation}
 d^{l+r+1, l+r+1}(S)=d^{l+r+1, r+1}(S_{l+1},S_{l+2}, \cdots , S_{l+r+1})
 \end{equation} where $S:=\left( S_{1},S_{2}, \cdots , S_{l+r+1}\right)$. However, for \(s\leq l,\)
\begin{eqnarray}\nonumber
&\ker d^{l+r,l+r}=\mathbb{R}(\0, v_{s}, \0, v_{r}, v_{r+1}, \cdots, v_{l+r-s}, \0_{s})+\oplus_{k=0, j=0}^{s-1, l+r-k}\mathbb{R}E_{l+r-k-j,j}\mathbf{e}^{l+r-k}_{l+r}&\\\label{kerlr}&
+\oplus_{k=0, i=1, j=0}^{s-1, 2, l-k}\mathbb{R}\Theta_{l-k-j,j}^i\mathbf{e}^{r+k}_{l+r}.\qquad\qquad\qquad\quad&
\end{eqnarray} Since \(\left(S_{1}, S_{2}, \cdots , S_{l+r} \right)\in \ker d^{l+r,l+r},\) \(d^{l+r+1,l+r+1}(S)\) for \(l\geq s\) is given by
\begin{eqnarray*}
&\left[S_{l+r-s+1},v_{s} \right]+\alpha\left[v_{s},v_{l+r-s+1} \right]+\alpha\sum_{k=r}^{l+1}\left[v_{k},v_{l+r-k+1} \right]=\left[S_{l+r-s+1},v_{s} \right]-\alpha\left[v_{l+r-s+1},v_{s} \right].\qquad\qquad\qquad\qquad\qquad\qquad&
\end{eqnarray*} The latter belongs to \({\rm im}\, d^{l+r+1,s+1}\). The equality here is followed from
\bas
&\sum_{k=r}^{l+1}\left[v_{k}, v_{l+r-k+1} \right]=\dfrac{1}{2}\sum_{k=r}^{l+1}\left( \left[v_{k},v_{l+r-k+1} \right]+\left[v_{l+r-k+1}, v_{k}\right]\right)=0.&
\eas For \(s>l+1,\) \(d^{l+r+1,l+r+1}(S)=\left[S_{l+r-s+1}, v_s \right] \subseteq {\rm im}\, d^{l+r+1,s+1}.\) When \(s=l+1,\)
\bes d^{l+r+1,l+r+1}(S)= \left[S_{r}, v_s\right]+\left[S_{s}, v_r\right]\subseteq {\rm im}\,d^{l+r+1,r+1}.\ees

For sufficiently large values of \(l,\) we merely consider equation \eqref{kerlr}. Thereby, \(\ker d^{l+r,l+r}\) has three subspaces. On the one hand, the subspaces
\(\oplus_{k=0, j=0}^{s-1, l+r-k}\mathbb{R}E_{l+r-k-j,j}\mathbf{e}^{l+r-k}_{l+r}\) and
\(\oplus_{k=0, i=1, j=0}^{l, 2, l-k}\mathbb{R}\Theta_{l-k-j,j}^i\mathbf{e}^{r+k}_{l+r}\) converge to zero in filtration topology when \(l\) approaches to infinity. On the other hand, the limit of the space \(\mathbb{R}(\0, v_{s}, \0, v_{r}, v_{r+1}, \cdots, v_{l+r-s}, \0_{s})\) in the filtration topology generates the vector field \(v^{(\infty)}-v_0.\) Since \(v^{(\infty)}-v_0\) is the trivial nonlinear symmetry transformation generator for \(v^{(\infty)},\) the proof is complete.
\end{proof}
    %%%%%%%%%%%%%%%%%%%%%%%%%%%%%%%%%%%%%%%%%%%%%%%%%%%%%%%%
\begin{cor} Assume that ${b_{0,1}}^2 b_{2,0}-b_{0,1}b_{1,0}b_{1,1}+b_{0,2}{b_{1,0}}^2\neq 0$ hold for the first level normal form coefficients in equation \eqref{Eq2}. When \(b_{1,0}\neq 0,\) the infinite level normal form of $v^{(1)}$ is given by
\begin{eqnarray}\label{inf}
&\Scale[0.9]{v^{(\infty)}\!=\!v_0\!+\!\sum^{2,1}_{i=1, j= 0}\!a^i_{1-j,j}\Theta_{1-j,j}^i\!+\!\sum^2_{i+j=1}b_{i,j}E_{i,j}\!+\!\sum_{j\geq 2}\!\left( a^1_{0,j}\Theta_{0,j}^1\!+\!a^2_{0,j}\Theta_{0,j}^2\!+\!b_{0,j+2}E_{0,j+2} \right).}&
\end{eqnarray} For \(b_{1, 0}=0,\) \(b_{0,1}\neq 0,\) the infinite level normal form is given by
\begin{eqnarray}\label{inf3}
\!&\Scale[0.9]{v^{(\infty)}\!=\!v_0\!+\!\sum^{2,1}_{i=1, j= 0}\!a^i_{1-j,j}\Theta_{1-j,j}^i\!+\!\!b_{0,1}E_{0,1}
\!+\!\sum_{j\geq 2}\!\left( a^1_{j, 0}\Theta_{j, 0}^1\!+\!a^2_{j, 0}\Theta_{j, 0}^2\!+\!b_{j+2, 0}E_{j+2, 0} \right).}&
\end{eqnarray}
The infinite level normal form coefficients \(b_{i, j}\) for \(i+j\leq 2\) are the same as in equation \eqref{Eq2}. Furthermore, \(v^{(\infty)}\) does not have any nontrivial nonlinear symmetry transformation generator within \(\LST.\)
\end{cor}
\begin{proof} %Since $\sum_{k=0}^{1}\sum^1_{j=0} {b_{1-k,k}}^2\,{b_{2-j, j}}^2\neq 0$
Assume that \(b_{1,0}\neq 0\). Then  \(s=1\) and \( p=0\). Since ${b_{0,1}}^2 b_{2,0}-b_{0,1}b_{1,0}b_{1,1}+b_{0,2}{b_{1,0}}^2\neq 0$ and \(E_{2-j, j}\)-terms for \(0\leq j\leq2\) cannot be normalized by the second level normal form, we have \((s, r)=(1, 2).\) Now  by Theorem \ref{s+1level}, \(\IM\, d^{l, 2}= \Span\,\{ E_{l-j, j}, \Theta^i_{l-j, j}|\, 0\leq j\leq l-1\}\) for \(l>2\), and
\begin{eqnarray*}
&v^{(2)}=v_0\!+\!\sum^{2,1}_{i=1, j= 0}a^i_{1-j,j}\Theta_{1-j,j}^i\!+\!\sum^2_{i+j=1} b_{i,j}E_{i,j}\!+\!\sum^\infty_{j=2}\left( a^1_{0,j}\Theta_{0,j}^1+a^2_{0,j}\Theta_{0,j}^2\!+\!b_{0,j+1}E_{0,j+1} \right).&
\end{eqnarray*} Now by Proposition \ref{3.7Rem}, \(3\leq\rank\,d^{3, 3}\leq 4.\) Due to the condition \({b_{0,1}}^2 b_{2,0}-b_{0,1}b_{1,0}b_{1,1}+b_{0,2}{b_{1,0}}^2\neq 0\), the three column vectors \((0, b_{2, 0}, b_{1, 1}, b_{0, 2})^t,\) \((0, b_{1, 0}, b_{0, 1}, 0)^t\) and \((0, 0, b_{1, 0}, b_{0, 1})^t\) are linearly independent. Hence, each column of the matrix \(\mathcal{M}^1_2\) is linearly independent with column space of \(\mathcal{M}^2_1\). Therefore, \(\rank\, d^{3, 3}=4\) and Theorem \ref{InftLevel} implies that
\begin{eqnarray*}
&\IM\,  d^{l, 3}= \IM\,  d^{l, 2}+ \Span\{ E_{3-j, j}|\, 0\leq j\leq 3\} \quad \hbox{ for }\; l\geq3.&
\end{eqnarray*} Theorem \ref{InfNFCor} concludes that the third level normal form \eqref{inf} is an infinite level normal form.

Now consider the case \(b_{1, 0}=0,\) \(b_{0, 1}\neq 0\). Then \(p=1.\) Similar to the above argument and by Theorem \ref{s+1level},
\(\IM d^{l, 2}= \Span\{ E_{l-j, j}, \Theta^i_{l-j, j}| 1\leq j\leq l\}\) for \(l\geq3\) and \(v^{(2)}\) follows
\begin{eqnarray*}
&v_0+\sum^{2,1}_{i=1, j= 0}a^i_{1-j,j}\Theta_{1-j,j}^i+\sum^2_{i+j=1}b_{i,j}E_{i,j} +\sum_{j\geq 2}\left( a^1_{j, 0}\Theta_{j, 0}^1+a^2_{j, 0}\Theta_{j, 0}^2+b_{j+1, 0}E_{j+1, 0} \right).\quad&
\end{eqnarray*} By Proposition \ref{3.7Rem} and the fact that each column of the matrix \(\mathcal{M}^1_2\) is linearly independent with column space of \(\mathcal{M}^2_1,\) we have \(\rank\,d^{3, 3}=4.\) Next, Theorem \ref{InftLevel} concludes that  \(E_{3-j, j}\)-terms for \(0\leq j\leq 3\) are simplified in the third level normal form and the third level normal form \(v^{(3)}\) is given by equation \eqref{inf3}. Finally by Theorem \ref{InfNFCor}, \(v^{(3)}\) is an infinite level normal form.
\end{proof}

\begin{thm}[Normal form classification when \(r=\infty\)]\label{ThmRinfty} Consider equations \eqref{sp} and \eqref{rq}. Then, the following holds.
\begin{itemize}
\item[1.] When \(r=\infty\) and \(s<\infty,\) the \(s+1\)-level normal form  \(v^{(1)}\) is given by
\begin{eqnarray}\label{s+1-level r=infty}
&v_0\!+\!b_{s-p,p}E_{s-j,j}\!+\!\sum_{l=0, i=1, j=0, j=l+p+1}^{\infty, 2, p-1, l+s} a^i_{l+s-j, j}\Theta_{l+s-j, j}^i\!+\!\sum_{l+j= 1, i=1}^{s, 2}a^i_{l, j}\Theta_{l, j}^i.\;&
\end{eqnarray}
When \(\sum_{i=1, k=1, j=0}^{2, \infty, k}{a^i_{k-j, j}}^2=0,\) the equation \eqref{s+1-level r=infty}  is infinite level normal form. The time-one map flows associated with vector fields from \(\Span\{E_{s-j, j}| 0\leq j\leq s\}\) are nonlinear symmetry transformations in the symmetry group of \(v^{(\infty)}.\)
\item[2.] For \(r=s=\infty,\) there exist \(s_i, p_i\) so that the infinite level normal form is given by either \(v_0,\)
\begin{eqnarray}\label{s+1sinftys1<infty}
&v_0+a^1_{s_1-p_1, p_1}\Theta_{s_1-p_1, p_1}^1+\sum_{l=0, j=0, j=l+p_1+1}^{\infty, p_1-1, l+s_1} a^1_{l+s_1-j, j}\Theta_{l+s_1-j, j}^1+\sum_{i+j=1}^{\infty}a^2_{i, j}\Theta_{i, j}^2\qquad&
\ea for \(s_1<\infty,\) or
\ba\label{s+1s2<infty}
&
v^{(s_2+1)}=v_0+a^2_{s_2-p_2, p_2}\Theta_{s_2-p_2, p_2}^2+\sum_{l=0, j=0, j=l+p_2+1}^{\infty, p_2-1, l+s_2} a^2_{l+s_2-j, j}\Theta_{l+s_2-j, j}^2,\quad &
\end{eqnarray} when \(s_1=\infty\) and \(s_2<\infty.\) The time-one map flows associated with vector fields from \(\rad\, \LST\) constitute nonlinear symmetry transformations in the symmetry group of \(v^{(\infty)}.\)
\end{itemize}
\end{thm}
\begin{proof}
Item 1. Theorem \ref{s+1level} results the \(s+1\)-th level normal form follows equation \eqref{s+1-level r=infty}. Let \({a^i_{k-j, j}}=0\) for all indices. When \(l<s\),
\begin{eqnarray*}%\mathbb{R}(\0_{s-1}, v_{s}, v_{s+1}, \cdots, v_{l+s})\oplus
&\ker d^{s+l, s+l} =  \sum_{k=1, j=0}^{s, l+k} \mathbb{R}(\0_{l},\0_{k-1}, E_{l+k-j, j},\0_{s-k})+\sum_{k=0, j=0}^{l, k} \mathbb{R}(\0_{s-1},\0_{k}, \Theta^i_{k-j, j},\0_{l-k}).&
\end{eqnarray*}
If \(l\geq s\),
 \begin{eqnarray}\label{kerlrrinft}
&\Scale[0.95]{\ker\,d^{s+l, s+l}= \sum_{k=1, j=0}^{s, l+k} \mathbb{R}(\0_{l},\0_{k-1}, E_{l+k-j, j},\0_{s-k})+ \sum_{j=0}^{s} \mathbb{R}(\0_{s-1}, E_{s-j, j},\0_{l-s},\0_s)}&\nonumber\\
&\sum_{k=1, j=0}^{s, l-s+k} \mathbb{R}(\0_{l},\0_{k-1}, \Theta^i_{l-s+k-j, j},\0_{s-k})&
\end{eqnarray}

By choosing \((\0_l, S^s_{l+1}, \cdots, S^s_{l+s})\in\ker\,d^{s+l, s+l}\) for \(l< s\) and \(c_{l+1-j, j}=0\) for all \(0\leq j\leq l+1\) in equation \eqref{Yslk}, we have
\begin{eqnarray*}
&\Scale[0.91]{d^{s+l+1, s+l+1}(\0_l, S^s_{l+1},\cdots, S^s_{l+s+1} )\!=\!\left[\sum_{j=0,i=1}^{l-s+1, 2}\!d^i_{l-s+1-j, j}\Theta^i_{l-s+1-j, j}, \sum_{j=p}^{s}\!a_{s-j, j}E_{s-j, j}\right]\!\in\! \im\, d^{s+l+1, s+1}.}&
\end{eqnarray*} For the case \(l\geq s\)
\begin{eqnarray*}
&\Scale[0.93]{d^{s+l+1, s+l+1}(\0_{s-1},S^s_s,\0_{l-s}, S^s_{l+1},\cdots, S^s_{l+s+1} )=\left[\sum_{j=0,i=1}^{l-s+1, 2}d^i_{l-s+1-j, j}\Theta^i_{l-s+1-j, j}, \sum_{j=p}^{s}a_{s-j, j}E_{s-j, j}\right].}&
\end{eqnarray*} The latter belongs to \(\im\, d^{s+l+1, s+1}.\) The argument for nonlinear transformation generators for the symmetry group of \(v^{(\infty)}\) is similar to the proof in Theorem \ref{InfNFCor}. The filtration topology-limit of the space \(\ker\,d^{s+l, s+l-1}\) given by equation \eqref{kerlrrinft} equals
\bas &\sum_{j=0}^{s} \mathbb{R}(\0_{s-1}, E_{s-j, j},\0_{l}).&\eas
Then, the time-one map flow associated with \(E_{s-j, j}\) commutes with that of \(v^{(\infty)}.\) This concludes the proof.

Item 2. For \(i=1, 2,\) define
\begin{eqnarray}\label{s1p1s2p2}
s_{i}:=\min \left\lbrace m\geq 1 \vert\, \exists\, j\leq m,\ a^i_{m-j,j}\neq 0 \right\rbrace,\quad p_i:=\min \left\lbrace j \vert\, a^i_{s_i-j,j}\neq 0 \right\rbrace.
\end{eqnarray} Hence,
\begin{eqnarray}
&\im\,d^{s_1+l+1, s_1+1}=\Span_{0\leq j\leq l+1}\{[E_{l+1-j,j}, \sum_{j=p_1}^{s_1}\Theta^1_{s_1-j, j}]\}.&
\end{eqnarray} Therefore, the \(s_1+1\)-th level normal form of \(v^{(1)}\) is expressed by equation \eqref{s+1sinftys1<infty} where \(\Theta^1_{l+s_1+1-j, j}\)-terms for \(p_1\leq j\leq l+p_1\) are simplified. Now we have
\begin{eqnarray}\label{Kers1}
 &\Scale[0.88]{\ker d^{s_1+l, s_1+l}\!=\!\sum_{k=0,j=0}^{s_1,l+k}\!\mathbb{R}E_{l+k-j, j}\mathbf{e}^{l+k}_{s_1+l}\!+\!\sum_{k=1,j=0}^{l+s_1,k}\!\mathbb{R}\Theta^1_{k-j, j}\mathbf{e}^{k}_{l+s_1}\!+\!\sum_{k=1, j=0}^{l,k-s_1}\!\mathbb{R}\Theta^2_{k-j, j}\mathbf{e}^{s_1+k}_{l+s_1}}&
\end{eqnarray} and this gives rise to
\begin{eqnarray*}
  &\im\,d^{s_1+l+1, s_1+l+1}=\Span_{0\leq j\leq l+1}\left\{\left[E_{l+1-j,j}, \sum_{j=p_1}^{s_1}\Theta^1_{s_1-j, j}\right]\right\}=\im\,d^{s_1+l+1, s_1+1}.&
\end{eqnarray*} Hence, no further terms can be normalized in the \(s_1+l\)-level. By the equation \eqref{Kers1}, the limit of the space \(\ker d^{s_1+l, s_1+l}\) in filtration topology equals to the linear space spanned by all nonlinear rotational vector fields. Therefore, the time-one map flows associated with the transformation generators from the radical ideal of \(\LST\) commutes with the flow of the simplest normal form.

Let \(s_1= \infty\) and \(s_2<\infty\) and define the grading function \(\delta(\Theta^i_{m, n})=\delta(E_{m ,n})=m+n.\)
This proves equation \eqref{s+1s2<infty}. Now the relation \(\im\,d^{s_2+l+1, s_2+l+1}=\im\,d^{s_2+l+1, s_2+1}\) for all \(l\geq 1\) holds due to
\begin{eqnarray*}
&\ker d^{s_2+l, s_2+l}\!=\!\sum_{k=1,j=0}^{s_2,l+k}\!\mathbb{R}(\0_{l+k-1}, E_{l+k-j, j},\0_{s_2-k})\!+\!\sum_{k=1,i=1,j=0}^{l+s_2,2,k}\mathbb{R}( \0_{k-1}, \Theta^i_{k-j, j},\0_{l+s_2-k}).&
\end{eqnarray*} This concludes that equation \eqref{s+1s2<infty} is an infinite level normal form. The argument for the nonlinear symmetry group of this case is similar to the case \(s_1<\infty\).
\end{proof}

%We remark that the conditions $\sum_{k=0}^{1}\sum^2_{j=0} {b_{1-k,k}}^2\,{b_{2-j, j}}^2\neq 0$ and ${b_{0,1}}^2 b_{2,0}-b_{0,1}b_{1,0}b_{1,1}+b_{0,2}{b_{1,0}}^2\neq 0$ represent the normal form coefficients \(b^{(2)}_{0, 2}\) and \(b^{(2)}_{2, 0}\) in the simplest orbital normal form while all other . The generic conditions guarantee that the cubic and quartic orbital normal form coefficients are not fully normalized. We have $$b^{(2)}_{0, 2}={b_{10}}^{-2}({b_{01}}^2b_{20}-b_{01}b_{10}b_{11} + {b_{10}}^2b_{02}),\, b^{(2)}_{2, 0}={b_{01}}^{-2}({b_{01}}^2b_{20}-b_{01}b_{10}b_{11} + {b_{10}}^2b_{02})$$

\section{Orbital normal form classification } \label{Sec4}

We define a module structure for time rescaling calculations. The integral domain of formal power series generated by \bes
Z_{m,n}:=\left({x_1}^2+{y_1}^2\right)^m\left({x_2}^2+{y_2}^2\right)^n \;\hbox{ for } \; m, n\geq0
\ees is denoted by \(\RScr\) and correspond with the near-identity time rescaling generators. Hence, \(\RScr\) acts on \(\LST\) by
\be\label{3rdlemma}
Z_{m,n}\Theta_{i,j}^1:=\Theta_{i+m,j+n}^1, \qquad Z_{m,n}\Theta_{i,j}^2:=\Theta_{i+m,j+n}^2, \qquad Z_{m,n}E_{i,j}:=E_{i+m,j+n},
\ee and $\mathscr{L}$ is a torsion free $\RScr$-module. Recall equations \eqref{dkr}-\eqref{dkr1}-\eqref{TransOrder} by introducing
\bes\mathds{B}:=\LST,\quad \mathds{A}:= (\RScr, [\LST, \LST]),\quad\hbox{ and }\;\; (T, S)* v:= Tv+ [S, v]\quad \hbox{ for }\; T\in \RScr, S\in[\LST, \LST].\ees We reorder time and state transformation generators in equation \eqref{TransOrder} so that time rescaling generators appear consecutively. Recall \(s\) as in equation \eqref{sp} and update the grading function $\delta$ by
\begin{equation}\label{GradingOrb}
\delta(E_{m,n}) =m+n, \ \delta(\Theta_{m,n}^i)= s+m+n, \quad \hbox{ for }\; i=1, 2.
\end{equation}

\begin{lem}\label{1-st Lem Orbital NormalForm}
There exist a sequences of permissible time scaling and changes of state variables that they transform the vector field \(v\) in equation \eqref{Eq1} into the $(s+1)$-th level extended partial orbital normal form
\begin{eqnarray*}
&v^{(s+1)} := v_0+ \sum_{i+j=s} a_{i, j}^1\Theta^1_{i, j}+ \sum_{i+j= 1}^{s} a^2_{i, j}\Theta_{i, j}^2 +\sum_{j=p}^{s} b_{s-j, j}E_{s-j,j} \qquad\qquad\qquad\qquad&\\
&+\sum_{l=1,j=0,j=l+p+1}^{\infty,p-1,l+s}b_{l+s-j,j}E_{l+s-j, j}+ \sum_{l=1,j=0,j=l+p+1}^{\infty,p-1,l+s} a^2_{l+s-j, j}\Theta_{l+s-j, j}^2.&
\end{eqnarray*}
\end{lem}
\begin{proof}
For $l\geq 1,$ we have
\begin{eqnarray*}
&\Scale[0.9]{\ker d^{l+s-1,s}\!=\!\Span_{k=0, j=0}^{s-1, l+k}\! \left\lbrace (\0_{k},\!Z_{l+k-j,j},\!\0),\! (\0_{k+s}, E_{l+k-j,j},\! \0) \right\rbrace
\!+\!\Span_{k=0,j=0, l+k-s\geq 0}^{s-1, l+k-s}\!\left(\0_{k+s},\! \Theta_{l+k-s-j,j}^i,\! \0\right).}&
\end{eqnarray*}
Thereby, \(\{d^{l+s,s+1}(T, \0)| (T, 0)\in \ker d^{l+s-1,s}\}+\rad\,\LST\) is given by
\bas
&\{d^{l+s,s+1}(\0, S, 0)| (\0, S)\in \ker d^{l+s-1,s}\}+\rad\,\LST \qquad\hbox{ for } \quad l\neq s,&
\eas and
\bas
&\Span\{\Theta^1_{l-k, k}| k=0, \cdots l\}\subseteq\{d^{l+s,s+1}(T, \0)| (T, 0)\in \ker d^{l+s-1,s}\}\cap\rad\,\LST.&
\eas For the case \(l\neq s\), we use the latter inclusion to simplify all \(\Theta^1\)-terms except those of grade \(2s.\) When \(l=s\), \(\{d^{2s,s+1}(\0, S, 0)| (0, S)\in \ker d^{2s-1,s}\}=\{0\}.\) Therefore, we instead use the inclusion
\bes\Span\{E_{2s-k, k}| k=p, \cdots, p+s\}\subseteq\Pi(\{d^{2s,s+1}(T, \0)| (T, 0)\in \ker d^{2s-1,s}\}+\rad\,\LST)\ees to simplify \(E_{2s-k, k}\)-terms for \(k=p, \cdots, p+s.\)
Since dimension of \({\rm im}\,d^{l+s, s}+\rad\, \LST\) is \(l+1\) as a subspace of \(\frac{\LST_{l+s}}{\rad\, \LST},\) \ie \({\rm im}\,d^{l+s, s}+\rad\, \LST\subseteq \frac{\LST_{l+s}}{\rad\, \LST},\) hypernormalization of \(E_{l+s-j, j}\)-terms beyond what are simplified in the non-orbital normal form process is not possible when \(l\neq s\). Given \(s\)-number of real values \(b_{2s-j,j}\in \mathbb{R}\) \((p\leq j\leq p+s)\) for \(l=s\), we introduce $\bar{c}_{s-j,j}$ in equation \eqref{Tl} as \(\bar{c}_{s, 0}:=\frac{-b_{2s-p,p}}{b_{s-p,p}},\) \(\bar{c}_{s-j, j}:= \frac{b_{2s-p-j,p+j}+\sum_{i=0}^{j-1}\bar{c}_{s-i, i} b_{s-p-j+i,p+j-i}}{-b_{s-p,p}}\) for \(1\leq j \leq s-p,\) and
\begin{eqnarray}\label{RecurRelations}
&\bar{c}_{s-j, j}:= \frac{b_{2s-p-j,p+j}+\sum_{j=0}^{s-p-1}\bar{c}_{2s-j-p-i, j-s+p+i} b_{j,s-j}}{-b_{s-p,p}}\;\; \hbox{ when }\; s-p+1\leq j \leq s.&
\end{eqnarray} Then, $d^{2s,s+1}\left(\sum_{j=0}^{s} \bar{c}_{s-j,j}Z_{s-j,j},\0_{2s+1}\right)= -\sum_{j=p}^{p+s}b_{2s-j,j}E_{2s-j,j},$ \ie terms $b_{2s-j,j}E_{2s-j,j}$ for $j=p,\cdots, p+s$ are simplified from the vector field $v^{(s+1)}.$
When \(l\neq s,\) let
\bas
&\bar{c}_{l-j,j}:=-\frac{1}{\omega_1} a^1_{l-j, j} \quad\hbox{ for }\quad 0\leq j \leq l, \; l\geq 1, &
\eas
while \( c_{l-j, j}\) and \(d^2_{l-j, j}\) are taken as in \(s+1\)-th level non-orbital hypernormalization. Hence, all \( \Theta^1_{l-j, j}\)-terms for all \(l\neq s,\) \(0\leq j\leq l\) can be simplified in \( s+1\)-th level orbital normal form. However, \(\Theta^2_{l-j, j}\)-terms are not simplified more than what is done in the non-orbital normal form case.
\end{proof}

%%%%%%%%%%%%%%%%%%%%%%%%%%%%%%%%%%%%%%%%%%%%%%%%%%%%%%%%%

\begin{thm}\label{r=infty ONF thm}
\begin{itemize}
  \item[1.] Let \(r=\infty\) and \( s<\infty.\) Then, the infinite level orbital normal form is given by
\begin{eqnarray*}
&\Scale[0.9]{\!v^{(s+1)}\!=\! v_0\!+\!\sum_{i+j=s}\!a_{i, j}^1\Theta^1_{i, j}\!+\!\sum_{i+j= 1}^{s}\!a^2_{i, j}\Theta_{i, j}^2\!+\!\sum_{j=p}^{s}\!b_{s-j, j}E_{s-j,j}
\!+\!\sum_{l=1,j=0,j=l+p+1}^{\infty,p-1,l+s}\!a^2_{l+s-j, j}\Theta_{l+s-j, j}^2.}&
\end{eqnarray*}
  \item[2.] When \(r=s=\infty\) and \(s_2<\infty\), the \(s_2+1\)-th level (and the infinite level)
orbital normal form of \(v^{(1)}\) is given by equation \eqref{s+1s2<infty}.
\end{itemize}
\end{thm}
\begin{proof} Item 1. In this case, the prime goal for normalization is to eliminate \(\Theta\)-terms while we prevent the creation of Eulerian terms. Since any further use of time rescaling beyond the first level is not possible, the proof follows the arguments in the \((s+l)\)-th level non-orbital normal form case.

Item 2. Since \(r=s=\infty,\) all Eulerian terms are simplified in the first level normal form. The centralizer of \(\mathbb{R}\Theta^1_{0, 0}+\mathbb{R}\Theta^2_{0, 0}\) in \(\LST\) is given by \(\LST=C_{\mathbb{R}\Theta^1_{0, 0}+\mathbb{R}\Theta^2_{0, 0}}(\LST).\) Hence, only time rescaling terms can be used for further simplification in the first level.
Hence all \(\Theta^1_{l-j, j}\)-terms for \(l\geq 1\) and \(0\leq j\leq l\) are simplified in the first level and the first level orbital normal form of \(v^{(1)}\) is represented by \(v^{(1)} := v_0+ \sum_{i+j=1}^{\infty} a^2_{i, j}\Theta_{i, j}^2.\) Now update the grading
function \(\delta(\Theta^i_{m, n})=\delta(E_{m, n})=m+n\) for \(i=1, 2\). Define \(s_i, p_i\) as in \eqref{s1p1s2p2}. Here this is, of course, defined for the first level orbital normal form. The linear space \(\ker d^{l+s_2-1, s_2}\) is given by
\begin{eqnarray}\label{Im d(l+s+1, s2+1) s=infty,s2<infty}
 & \sum_{k=0,j=0}^{s_2-1, l-s+k}\left\{\mathbb{R}(\0_{s_2+k}, E_{l-s_2+k-j, j}, \0_{s_2-k-1})
 +\sum_{i=1}^{ 2}\mathbb{R}(\0_{s_2+k}, \Theta^i_{l-s_2+k-j, j}, \0_{s_2-k-1})\right\},& \nonumber\\
&\hbox{ and }\quad \im\,d^{l+s_2+1, s_2+1}= \sum_{j=0}^{l}\mathbb{R}\left[E_{l-j, j}, \sum_{j=0}^{s_2}\Theta_{s_2-j, j}^2\right]+\sum_{j=0}^{l}\mathbb{R}Z_{l+s_2-j, j}v_0.&
\end{eqnarray} Since our priority of elimination is with \(\Theta^1\)-terms than \(\Theta^2\)-terms, no time rescaling can be used for further elimination. Thereby, equation \eqref{s+1s2<infty} represents \((s_2+1)\)-level normal form of \(v^{(1)}\). Now we have
 \begin{eqnarray*}
  \ker d^{l+s_2,l+s_2}=&\sum_{k=1,j=0}^{s,l+k}\mathbb{R}E_{l+k-j,j}\mathbf{e}^{2l+s_2+k}_{2l+2s_2}+\sum_{k=1,j=0}^{l+s_2,k}\mathbb{R}\Theta^i_{k-j, j}\mathbf{e}^{l+s_2+k}_{2l+2s_2}.&
 \end{eqnarray*} Thus, \(\im\,d^{l+s_2+1, l+s_2+1}=\im\,d^{l+s_2+1, s_2+1}.\) Hence, \eqref{s+1s2<infty} is the infinite level orbital normal form.
\end{proof}
% end color
%%%%%%%%%%%%%%%%%%%%%%%%%%%%%%%%%%%%%%%%%%%%%%%%%%%%%%%%%
The following lemma plays a central role in the infinite level derivation of orbital normal forms.
\begin{lem}\label{lemRank}
Let \(\rank\left[\mathcal{M}_r^s \quad \mathcal{M}_s^r  \right]=\alpha,\) \(s, r\) be similarly defined as in equations \eqref{sp}-\eqref{rq}, and \(l\geq 0\). Then,
\begin{equation}\label{RankThm}
\rank\left[\mathcal{M}_r^{l} \quad \mathcal{M}_s^{l+r-s}\right]=
\left\{\begin{array}{lcc}
2l+r-s+2 & \text{for}  & 0 \leq l \leq \alpha-r-2,\\ [0.7ex]
\alpha+l-s & \text{when }  & l > \alpha-r-2.
\end{array}\right.
\end{equation} The case \(l=0\) is useful for the results in Section \ref{Sec5}.
%Furthermore, for \(d= r+2,\) \(\rank\left[\mathcal{M}_r^{l} \quad \mathcal{M}_s^{l+r-s}\right]=d+l-s\) for all \( l \geq 1 \).
%where \( M=d-r-1 \).% In the case that \( q>p \), let \( M=d-s-1 \).
\end{lem}

%%%%%%%%%%%%%%%%%%%%%%%%%%%%%%%%%%%%%%%%%%%%%%%%%%%%%%%%

Now we update the grading function $\delta$ by
\begin{equation*}
\delta(E_{m,n}) =m+n, \ \delta(\Theta_{m,n}^i)= r+s+m+n, \ i=1, 2.
\end{equation*}

\begin{thm}\label{r+1Orb}
Let \(p>q\). There exist sequences of non-negative integers \(\dd_l,\) permissible state transformations and time scalings so that they transform equation \eqref{Eq1} into the \(r+1\)-th level orbital normal form
\begin{eqnarray*}
&v_0\!+\!b_{r-q, q}E_{r-q, q}\!+\!\sum_{i=1, l=1,j=1}^{2, s,l}\!a^i_{l-j, j}\Theta_{l-j, j}^i\!+\!\sum_{j=p}^{s}\! b_{s-i, i}E_{s-i,i}\!+\!\sum_{l=1, j=0,l+p+1}^{\infty, p-1,l+s}\!a^1_{l+s-j, j}\Theta_{l+s-j, j}^1&\\
&\!+\!\sum_{i+j=s} a_{i, j}^2\Theta^2_{i, j}\!+\!\sum_{l=0,j=0, l+p+r-s+\dd_l+1}^{\infty, q-1, l+r}b_{l+r-j, j}E_{l+r-j, j}\!+\!\sum_{ l=0}^{p-q-1}\sum_{j=q+l+1}^{p-1}b_{l+r-j, j}E_{l+r-j, j}&
\end{eqnarray*} where \(b_{r-j, j}=0\) for \(j=0, \ldots, q-1\) and \(\dd_l\leq\min\{s-p, \max\{l-p+q+1, 0\}\}.\)
\end{thm}
\begin{proof}
We have
\begin{eqnarray*}
&\ker d^{l+r-1, r}:=\Span^{s-1,l+r-s+k}_{k=0,j=0}\{(\0, Z_{l+r-s+k-j,j},\0), (\0_{k+r-s}, E_{l+r-s+k-j,j},\0_{k-s+1})\}&\\
&+\oplus_{k=0,j=0}^{r-s-1,l+k}\mathbb{R}(\0, 2(s-l-k)Z_{l+k-j,j}, \0, E_{l+k-j,j}, \0)+\oplus_{i=1, k=0, j=0}^{2, s-1,l-2s+k}\mathbb{R}(\0,\Theta_{l-2s+k-j,j}^i,\0).&
\end{eqnarray*} We refer to \(2(s-l-k)Z_{l+k-j,j}\mathbf{e}^{k+1}_{2l+2r}+E_{l+k-j,j}\mathbf{e}^{l+k+r+1}_{2l+2r}\) as a coupled term, that is a time term coupled with a state term. For an arbitrary \(\Scale[0.85]{(T, S)\!:=\!( T_l,\! T_{l+1},\! \cdots,\! T_{l+r-1},\! S_{l},\! \cdots,\! S_{l+r-1})\!\in\!\ker d^{l+r-1,r}\!,}\)
\bas
&d^{l+r, r+1}(T, T_{l+r}, S,  S_{l+r})=T_{l}v_r+ T_{l+r-s}v_s+ \left[ S_{l}, v_r \right]+ \left[ S_{l+r-s}, v_s \right],&
\eas
\ba\label{Yk}
&S_{k}:=\sum_{j=0}^{k} c_{k-j,j} E_{k-j, j}+\sum_{i=1, j=0}^{2, k-r-s} d_{k-r-s-j,j}^i \Theta^{i}_{k-r-s-j, j},&
\ea \(\bar{c}_{l+k-j, j}:=2(s-l-k)c_{l+k-j, j}\) for \(0\leq k\leq r-s-1\) and \(d_{l+k-r-s-j,j}^i:=0\) for \(0\leq k\leq r-s-1\). Hence,
\bas
&\im\,d^{l+r, r+1}\cap(\LST_{l+r}\cap\rad\,\LST)=\big\{\big[ \sum_{i=1}^{2}\sum_{j=0}^{l-2s} d_{l-2s-j,j}^i \Theta^{i}_{l-2s-j, j} , v_s \big]\big|\, d_{l-2s-j,j}^i\in \mathbb{R}\big\}&
\eas \ie \(\Theta\)-terms cannot be simplified more than what was simplified in the \((s+1)\)-th level orbital normal form. Hence we apply a reduction approach using the factor Lie algebra
\(\hat{\mathds{B}}:=\frac{\LST}{\rad\,\LST},\) \(\hat{\mathds{A}}= (\RScr, [\frac{\LST}{\rad\, \LST}, \frac{\LST}{\rad\, \LST}]),\)
and by inductively defining
\be
\label{dhat}\hat{d}^{l+r, r+1}: \ker(\hat{d}^{l+r-1,r})\times \hat{\mathds{A}}_{l+r}\rightarrow \hat{\mathds{B}}_{l+r}
\ee as a projection of the map \({d}^{l+r, r+1}\) on the factor Lie algebra. More precisely, we replace \((\hat{\mathds{A}}, \hat{\mathds{B}}_k, \hat{d}^{l+r, r+1})\) with \(({\mathds{A}}, {\mathds{B}}_k, {d}^{l+r, r+1})\) in \eqref{dkr}-\eqref{dkr1}-\eqref{TransOrder} in order to only discus further simplification of Eulerian terms in the factor Lie algebra \(\hat{\mathds{B}}\). Given the above argument, \({\rm im}\,{d}^{l+r, r+1}= {\rm im}\,\Pi\circ\hat{d}^{l+r, r+1}+ {\rm im}\,{d}^{l+r, s+1}.\) Similarly terms \(E_{l+k-m,m}\mathbf{e}^{l+r-s+k}_{2l+2r}+{\rad\,\LST}\) precede \(E_{l+k-j,j}\mathbf{e}^{l+r-s+k}_{2l+2r}+{\rad\,\LST}\) terms and
term \(Z_{l+k-m,m}\mathbf{e}^{k+1}_{2l+2r}\) precedes \(Z_{l+k-j,j}\mathbf{e}^{k+1}_{2l+2r},\) when \(m < j.\) Then, the matrix representation for \(\hat{d}^{l+r, r+1}\) is
\ba\label{MatrxDr+1}
&\left[2(s-r) \mathcal{M}_{r}^{l} \quad \mathcal{M}_{s}^{l+r-s} \quad 2(l+r-2s)\mathcal{M}_{s}^{l+r-s} \right].&
\ea Since the second and third blocks are linearly dependent, we omit the third block matrix by assigning \(c_{l+r-s-j, j}:=0\) in the transformation generators for all \(l\) and \(0\leq j\leq l+r-s\). Let \(l< p-q\). Since the first \(q\)-rows and \(p\)-rows of the matrices \(\mathcal{M}_{r}^{l}\) and \(\mathcal{M}_{s}^{l+r-s}\) are zero, \(E_{l+r-j, j}\)-terms for \(0\leq j\leq q-1\) cannot be normalized. On the other hand, terms of the form $E_{l+r-j,j}$ for $q \leq j \leq q+l$ are simplified via
\begin{small}\bas
&\sum^l_{j=0}\gamma_j\hat{d}^{l+r, r+1}\big(2(s-l)Z_{l-j,j}\mathbf{e}^1_{2l+2r}+E_{l-j,j}\mathbf{e}^{l+r+2}_{2l+2r}\big)+ \sum^{q+l}_{j=q} b_{l+r-j, j} E_{l+r-j, j}, &
\eas\end{small} that belongs to \(\sum^{l+r}_{j=q+l+1}\mathbb{R} E_{l+r-j, j}+\rad\,\LST.\) Here,
\begin{eqnarray}\label{RecurRelations}
&\gamma_{0}:=\frac{b_{l+r-p,p}}{b_{s-p,p}}, \qquad  \gamma_{i}:= \frac{b_{l+r-p-i,p+i}-\sum_{j=1}^{i}\gamma_{i-j} b_{s-p-j,p+j}}{2(s-r)b_{s-p,p}}\quad \hbox{ for }\;  1\leq i \leq l,&
\end{eqnarray} and real numbers \(b_{l+r-j, j}\) stand for the coefficients of the normalizing vector field. Here note that the transformation generated by \(\sum^l_{j=0}\gamma_j(2(s-l)Z_{l-j,j}\mathbf{e}^1_{2l+2r}+E_{l-j,j}\mathbf{e}^{l+r+2}_{2l+2r})\)
changes the coefficients associated with terms in the radical ideal and \(b_{l+r-j, j} E_{l+r-j, j}\)-terms for \(q+l+1\leq j\leq l+r\). Since \(l\leq p-q-1,\) the column spaces associated with matrices \(\mathcal{M}_{r}^{l}\) and \(\mathcal{M}_{s}^{l+r-s}\) are linearly independent and \(\dim \hat{d}^{l+r, r+1}= 2l+r-s+2.\) We omit the first \(p\)-rows and last \(s-p\)-rows of \([2(s-r) \mathcal{M}_{r}^{l}\quad \mathcal{M}_{s}^{l+r-s}]\) to obtain matrices \(A_l\) and \(B_l,\) respectively. Hence, \(E_{l+r-j, j}\)-terms for \(p\leq j\leq p+l+r-s\) (via equation \eqref{cbar}) and \(0\leq j\leq q-1\) (via equation \eqref{RecurRelations}) are simplified while \(E_{l+r-j, j}\)-terms for \(q\leq j\leq p-1\) cannot be normalized in the \(r+1\)-th level normalization step.

When $l+1\geq p-q,$ we use the recursive relations \eqref{RecurRelations} to eliminate $E_{l+r-j,j}$-terms for $q\leq j\leq p-1$ in the \(r+1\)-th level.
Remove the first $p$-rows and the first $p-q$-columns of \([2(s-r) \mathcal{M}_{r}^{l} \quad \mathcal{M}_{s}^{l+r-s}]\) to obtain a matrix blocked by \(A_l,\) \(B_l,\) \(C_l,\) \(D_l,\) where \(B_l\) is lower triangular with \(b_{s-p, p}\neq0\) on its diagonal entries. Let
\be\label{ul1}
\dd_l :=\rank (C_{l}-D_{l}B^{-1}_{l}A_{l})\quad \hbox{ for } l+1> p-q, \quad \hbox{and }\; \dd_l=0 \hbox{ for } l+1\leq p-q.
\ee Next by Lemma \ref{lemRank}, we have
\begin{equation}\label{ul2}
\dd_l=  \left\{ \begin{array}{lcc}
l+1-p+q        & \text{ when }       & p-q-1 \leq l \leq \alpha-r-2,\\ [0.7ex]
\alpha-r-1-p+q          & \text{for}       & l > \alpha-r-2,
\end{array}\right.
\end{equation} where \(\alpha:=\rank \left[ \mathcal{M}_r^s \quad \mathcal{M}_s^r \right]\). Hence, \(\rank\,\hat{d}^{l+r, r+1}= \dd_l+l+r-s+1+p-q\) and all terms $E_{l+r-j,j}$ for $p\leq j\leq p+l+r-s +\dd_l $ are simplified and
\bes
\Pi({\rm im}\, d^{l+r, r+1}+\rad\LST)=\Pi({\rm im}\, d^{l+r, s+1}+\Span_{j=q, j=p+l+r-s}^{k, p+l+r-s+\dd_l} \{E_{l+r-j, j}\}+\rad\LST),
\ees
for \(k=q+l\) when \(l<p-q,\) and \(k=p-1\) when \(l\geq p-q\).
\end{proof}

 %%%%%%%%%%%%%%%%%%%%%%%%%%%%%%%%%%%%%%%%%%%%%%%%%%%%%%%%%

\begin{rem} For \(p<q,\) we have
\({\rm im}\, d^{l+r, r+1}={\rm im}\, d^{l+r, s+1}+\Span \{E_{l+r-j, j}\,|\, p\leq j\leq k, p+l+r-s\leq j\leq p+l+r-s+\dd_l\},\) \(\dd_l=0\) and \(k:=p+l+r-s\) for \(l+r-s+p<q.\) When \(l+r-s\geq q-p\), \(\dd_l\leq\min\{s-p,l+1\}\) and \(k:=q-1\).
\end{rem}
\begin{prop}\label{cor4.4}
When \(\rank\, d^{r+s, r+1}=r+s+1\) and \(p>q\), the infinite level orbital normal form is
\begin{eqnarray*}
&v^{(r+1)}= v_0\!+\!b_{r, 0}E_{r, 0}\!+\!\sum_{i+j= 1}^{s} a^1_{i, j}\Theta_{i, j}^1 \!+\!\sum_{j=p}^{s} b_{s-i, i}E_{s-i,i}\!+\!\sum_{l=1, j=0, j=l+p+1}^{\infty, p-1, l+s} a^1_{l+s-j, j}\Theta_{l+s-j, j}^1&\\
&\!+\!\sum_{i+j= s} a^2_{i, j}\Theta_{i, j}^2\!+\!\sum_{l=0, j=l+1, p+l+r-s+1}^{p-1, p-1,l+r}b_{l+r-j, j}E_{l+r-j, j}\!+\!\sum_{l=p, j=2l+r-s+2}^{s-2, l+r}b_{l+r-j, j}E_{l+r-j, j}.&
\end{eqnarray*} %Furthermore, \(v^{(r+1)}\) constitutes an infinite level orbital normal form for \(v^{(1)}.\)
\end{prop}
\begin{proof} By Proposition \eqref{3.7Rem}, \(\rank\left[\mathcal{M}_{r}^s \quad \mathcal{M}_s^r \right] =r+s+1\) and \(p>q\) imply that \(q=0\). By Lemma \ref{lemRank}, Theorem \ref{r+1Orb}, \(m=s-1,\)
\begin{equation}
\dd_l=  \left\{ \begin{array}{lcc}
l+1-p        & \text{when }       & p-1\leq l\leq s-1,\\ [0.7ex]
s-p          & \text{for}       & l>s-1,
\end{array}\right. \quad\hbox{ and } \; \dd_l=0\, \hbox{ for }\, l< p-1.
\end{equation} For \(1\leq l\leq p-1,\) Theorem \ref{r+1Orb} implies that \(E_{l+r-j, j}\)-terms for \(0\leq j\leq l\) and \(p\leq j\leq p+l+r-s\) are normalized. However, \(E_{l+r-j, j}\)-terms for \(l+1\leq j\leq p-1\) and \(p+l+r-s+1\leq j\leq l+r\) cannot be simplified. In the cases of \(l\geq s-1,\) all Eulerian terms of grade \(l+r\) are normalized. Now we show that
\begin{equation}
d^{l+r+1, l+r+1}(T, S)=d^{l+r+1, r+1}(T_{l+1}, T_{l+2}, \cdots , T_{l+r+1}, S_{l+1},S_{l+2}, \cdots , S_{l+r+1})
\end{equation} for all \(l\geq 1,\) where $T=(T_{1}, T_{2}, \cdots , T_{l+r+1})\) and \(S:=(S_{1},S_{2}, \cdots , S_{l+r+1})$. Assume that \(1\leq l\leq s-2\). Since all terms \(E_{l+r-j, j}\) for \(l\geq s-1\) and \(0\leq j\leq l+r\) are simplified in \(r+1\)-th level, \(\ker d^{l+r, l+r}\) is given by
\begin{eqnarray*}
&\Scale[0.855]{\sum_{k=1, j=0}^{r-s, l+k} \mathbb{R}\frac{2(s-l-k)Z_{l+k-j, j}\mathbf{e}^{l+k}_{2l+2r}+ E_{l+k-j, j}\mathbf{e}^{2l+k+r}_{2l+2r}}{2}
\!+\!\Span_{k=1, j=0}^{s, l+r-s+k}\{E_{l+r-s+k-j, j}\mathbf{e}^{2l+k+2r-s}_{2l+2r}, Z_{l+r-s+k-j, j}\mathbf{e}^{l+k+r-s}_{2l+2r}\}.}&
\end{eqnarray*}
By Lemma \ref{lemRank} for \(l\geq s-1,\) \(\ker d^{l+r+1, l+r+1}\) is described as
\begin{eqnarray*}
&\Span_{k=1, j=0}^{s, l-r+s+k} \{\left(\0, Z_{l-r+s+k-j, j}, \0\right)\!,\!\left(\0, E_{l-r+s+k-j, j}, \0\right)\}\!+\!\sum_{k=1, j=0}^{s, l-2s+k}\mathbb{R}\Theta^i_{l-2s+k-j, j}\mathbf{e}^{2l+k+2r-s}_{2l+2r}.&
\end{eqnarray*} Hence for \(1\leq l\leq s-2,\)
\begin{equation*}
d^{l+r+1, l+r+1}\!\left(T\!,\!T_{l+r+1}\!,\!S,\!S_{l\!+\!r\!+\!1} \right)\!=\!T_{l\!+\!1}v_r+ T_{l+r-s\!+\!1}v_s+ \left[ S_{l+1},\! v_r \right]+ \left[ S_{l+r-s+1},\! v_s \right]\!\in\! {\rm im}\,d^{l+r+1, r+1},
\end{equation*} where $T_{l+k}, S_{l+k}$ are defined by equations \eqref{Yk}, \(\bar{c}_{l+k-j, j}\!=\!2(s-l-k)c_{l+k-j, j}\) for \(1\leq k\leq r-s,\) and \(d_{l+k-j,j}^i=0\) for \(1\leq k\leq s,\) \(0\leq j\leq l+k,\) \(i=1, 2\). Similarly for \(l\geq s-1,\) we have
\begin{equation*}
d^{l+r+1, l+r+1}\left(T, T_{l+r+1}, S, S_{l+r+1} \right)= \left[ S_{l+r-s+1}, v_s \right]\in {\rm im }\, d^{l+r+1, s+1}\subseteq{\rm im }\, d^{l+r+1, r+1}
\end{equation*} where \(c_{l+r-s+1-j, j}=0\) for all \(0\leq j\leq l+r-s+1\). This completes the proof.
\end{proof}
    %%%%%%%%%%%%%%%%%%%%%%%%%%%%%%%%%%%%%%%%%%%%%%%%%%%%%%%%%

\begin{thm}\label{s1r2}
When \(b_{1, 0}=0\) and \(b_{0, 1}b_{2,0}\neq0\) in equation \eqref{Eq2}, the second level extended partial orbital normal form of \(v^{(1)}\) is given by
\ba\label{s1r22nONFEq}
&\Scale[0.95]{v^{(2)}=v_0\!+\!\sum_{i+j= 1}\left(a^{1(2)}_{i,j}\Theta_{i,j}^1\!+\!a^{2(2)}_{i,j}\Theta_{i,j}^2+b_{i,j}E_{i,j}\right)\!+\!\sum^\infty_{j=2} b^{(2)}_{j, 0}E_{j, 0}+ \sum_{j\geq 2} a^{1(2)}_{j, 0}\Theta_{j, 0}^1, }&
\ea and \(b^{(2)}_{2, 0}=b_{2, 0}.\) Further, the infinite level orbital normal form of \(v^{(1)}\) is
\ba\label{s1r2 Third Level Orbital NF}
&v^{(3)}=v_0\!+\!\sum_{i+j= 1}\left(a^{1(2)}_{i,j}\Theta_{i,j}^1+a^{2(2)}_{i,j}\Theta_{i,j}^2+b_{i,j}E_{i,j}\right) +  b_{2, 0}E_{2, 0}+ \sum_{j\geq 2} a^{1(3)}_{j, 0}\Theta_{j, 0}^1.&
\ea
\end{thm}
\begin{proof} We have
\bes d^{2,2}\left(\frac{-b_{0, 2}}{b_{0, 1}}Z_{0, 1}+\frac{-b_{1,1}}{b_{0, 1}}Z_{1, 0},0\right)= -b_{0, 2}E_{0, 2}-b_{1,1}E_{1,1}\ees for any \(b_{0, 2}, b_{1,1}\in \mathbb{R}.\) Hence, terms of \(E_{0,2}\) and \(E_{1,1}\) are simplified in the second level orbital normal form while the normal form coefficient \(b^{(2)}_{2, 0}=b_{2, 0}\) remain unchanged. On the other hand, Theorem \eqref{1-st Lem Orbital NormalForm} implies that $E_{j, 0}$-terms \((2\leq j)\) are the only possible remaining terms in the second level orbital normal form. Thus, equation \eqref{s1r22nONFEq} is the second level extended partial orbital normal form. Since \(b^{(2)}_{2, 0}=b_{2, 0}\neq 0,\) \((s, r)=(1, 2),\) \(q=0,\) and \(p>q\). We may use $d^{l+2,3}(\frac{b_{l+2, 0}}{2 b_{2, 0}}(-2(l-1)Z_{l, 0},0, E_{l, 0},0))=-b_{l+2, 0}E_{l+2, 0}$ to simplify \(E_{j, 0}\)-terms for \(j\geq 3.\)
These give rise to the normalization of all Eulerian terms of grade \(l+2\). Thus, equation \eqref{s1r2 Third Level Orbital NF} represents the third level orbital normal form of \(v^{(2)}\) according to Theorem \ref{r+1Orb}. The block matrices \(A_l,\) \(B_l,\) \(C_l\) and \(D_l\) are obtained by removing the first column and row of $\left[  -2\mathcal{M}_{2}^{l} \quad \mathcal{M}_{1}^{l+1} \right]$. Hence, \(\rank[\mathcal{M}^2_3\quad\mathcal{M}^3_2]=4\) by Proposition \ref{3.7Rem}. Next, the third level orbital normal form \(v^{(3)}\) is the infinite level orbital normal form according to Proposition \ref{cor4.4}.
\end{proof}

\section{Multiple-input parametric normal forms} \label{Sec5}

We consider a multiple-parametric system given by equation \eqref{InpSys}. Using the primary shifts of coordinates \cite[page 373]{MurdBook}, we can eliminate the nonzero parameter-dependent constants from the system. Hence, we can assume that \(F(0, \mu)=G(0, \mu)=H(0, \mu)=0.\)
By formulas \eqref{ETheta} for the case \(v_0\), we obtain a parametric version of Lemma \ref{1stLevel}, \ie the first level extended partial parametric normal form of \eqref{InpSys} is given by
\begin{eqnarray}\label{eqp1}
&w^{(1)}=v_0+\sum_{j+j\geq 0}a^1_{j,k}(\mu) \Theta_{j,k}^1+\sum_{j+k\geq 0}a^2_{j,k}(\mu) \Theta_{j,k}^2+\sum_{j+k\geq 0}b_{j,k}(\mu)E_{j,k},&
\end{eqnarray} where \(\mu:=\left(\mu_1,\cdots,\mu_{N} \right)\) stands for the inputs, \(\m:= (m_1, m_2, \cdots, m_N),\) \(\mu^{\m}:={\mu}^{m_1}\ldots\mu_{N}^{m_N},\)
 and for \(i=1, 2,\)
\bas
&a^i_{j,k}(\mu):=\sum_{|\mathbf{m}|=0}^{\infty}\alpha_{j, k, \m}^i\mu^{\m},
b_{j,k}(\mu):=\sum_{\vert \mathbf{m}\vert=0}^{\infty}\beta_{j, k, \m}\mu^{\m}, \vert \mathbf{m}\vert:= m_1+\cdots + m_{N}, &
\eas and \(a^i_{0,0}(0)=b_{0,0}(0)=0.\) Now we assume that
\begin{equation}\label{ParS}
s:=\min\left\lbrace m\geq 1 \vert \exists j\geq 0,  b_{m-j,j}(0) \neq 0\right\rbrace, \quad p:=\min \lbrace i \vert\, b_{s-i, i}(0) \neq 0\rbrace, \quad p\leq s.
\end{equation}
    %%%%%%%%%%%%%%%%%%%%%%%%%%%%%%%%%%%%%%%%%%%%%

\begin{thm}\label{s+1PNFThm}
Given the vector field \eqref{eqp1} and conditions \eqref{ParS}, there exist time rescaling and changes of state variables transforming \(w^{(1)}\) into the \((s+1)\)-th level extended partial parametric normal form
\begin{eqnarray}\label{s+1PNF}
&\Scale[0.95]{ w^{(s+1)}:=v_0 +\sum_{i+j=0}^{s} \left(a^2_{i,j}(\mu)\Theta^{2}_{i, j}+ b_{i,j}(\mu)E_{i,j} \right)+\sum_{l=1, j=0, j=p+l+1}^{\infty, p-1, l+s}b_{l+s-j,j}(\mu)E_{l+s-j,j}}&\\\nonumber
&\Scale[0.95]{+\sum_{i+j=s} a^1_{i,j}(\mu)\Theta^{1}_{i, j} +\sum_{l=1, j=0, j=p+l+1}^{\infty, p-1, l+s}a^2_{l+s-j,j}(\mu)\Theta^{2}_{l+s-j,j},\;\,}&
\end{eqnarray} where $b_{s-p,p}(\mu)=b_{s-p,p}(\0)\neq 0$, \(b_{i,j}(\0)=0\) for \(i+j<s\), and \(a^i_{0,0}(\0)=0\) for \(i=1, 2\).
\end{thm}
\begin{proof} We use a structure constant extension to include
\bas
&\left[E_{k,l},\Theta_{0,0}^i \right]=0,\,  \left[E_{0,0}, E_{k,l}\right]=2(k+l)E_{k,l}, \,\,
\left[E_{0,0}, \Theta_{k,l}^i\right]=-2(k+l)\Theta_{k,l}^i,  &
\eas for all \(k, l\geq 0,\) \(i=1, 2.\) Thus, the parametric terms \(E_{s-p,p}\mu^\m\) for any nonzero \(\m\in \mathbb{Z}_{\geq0}^N\) is simplified in the \(s+1\)-th level parametric normal form. We simplify parametric terms \(a^1_{0, 0, \m}\Theta^1_{0, 0}\mu^\m\) from the system by time rescaling via
\begin{footnotesize}\bas
&\frac{-a^1_{0, 0, \m}}{\omega_1}Z^0_0\mu^\m v_0= -\frac{a^1_{0, 0, \m}\omega_2}{\omega_1}\Theta^2_{0, 0}\mu^\m-a^1_{0, 0, \m}\Theta^1_{0, 0}\mu^\m.&
\eas\end{footnotesize} By parametric version of the formulas given in the proof of Lemma \ref{1-st Lem Orbital NormalForm}, the proof is complete.
\end{proof} Let
\begin{equation}\label{ParRQ}
r:=\min \left\lbrace m>s \vert\, \exists j, b_{m-j,j}(0) \neq 0 \right\rbrace, \quad q:=\min \lbrace j \vert\, b_{r-j, j}(0) \neq 0\rbrace, \quad q\leq r.
\end{equation}

\begin{thm}\label{r+1PNFThm}
Consider \(s,\) \(r<\infty,\) \(p,\) \(q\) defined by equations \eqref{ParS} and \eqref{ParRQ}. Let \(q>p.\) Then, there exist a sequence of natural numbers \(\dd_{l}\) and invertible transformations (time scaling and changes of state variables) transforming $w^{(1)}$ in \eqref{eqp1} into the \((r+1)\)-th level extended partial parametric normal form
%\begin{footnotesize}
\begin{eqnarray*}
&\Scale[0.88]{ w^{(r+1)} := v_0+ \sum_{i+j= 1}^{s}\left( a^2_{i, j}(\mu)\Theta_{i, j}^2+ b_{i, j}(\mu)E_{i, j}\right) +\sum_{l=1, j=0, j=l+p+r-s+\dd_l+1}^{\infty, p-1, l+r}b_{l+r-j, j}(\mu)E_{l+r-j, j}}&\\
&\Scale[0.88]{+\sum_{l=1, j=p+l+r-s+1}^{q-p-r+s-1, q-1}b_{l+r-j, j}(\mu)E_{l+r-j, j}+\sum_{l=1,j=0,j=p+l+1}^{r-s, p-1, s+l}b_{l+s-j, j}(\mu)E_{l+s-j, j}}&\\
&\Scale[0.88]{+\sum_{i+j=s} a^1_{i,j}(\mu)\Theta^{1}_{i, j}+\sum_{l=1, j=0 j=l+p+1}^{\infty, p-1, l+s} a^2_{l+s-j, j}(\mu)\Theta_{l+s-j, j}^2.}\qquad\quad\qquad\;\;&
\end{eqnarray*}%\end{footnotesize}
Here, \(b_{l+s-j, j}(\0)=0\) when \(0\leq j\leq p-1\) and \(p+l+1\leq j \leq l+s,\) for \(1\leq l\leq r-s-1.\) Furthermore,  \(b_{s-p,p}(\mu)=b_{s-p,p}(0)\neq 0\) and \(b_{r-q,q}(\mu)=b_{r-q,q}(0)\neq 0.\)
\end{thm}
\begin{proof}
Since
\begin{equation*}
\mathbb{R}\left(\0_{r}, E_{0,0}\mu^{\m},\0_{r-1}\right)+ \mathbb{R}(2sZ_{0,0}\mu^\m,\0_{r-1}, E_{0,0}\mu^\m , \0_{r-1})\subseteq
\ker d^{r-1+(r+1)|\m|, r},
\end{equation*}
\begin{equation*}
d^{r+(r+1)|\m|, r+1}\left( \frac{sZ_{0,0}\mu^\m}{(r-s)b_{r-q,q}},\0_{r}, \frac{E_{0,0}\mu^\m}{2(r-s)b_{r-q,q}}, \0_{r} \right)= -\mu^\m E_{r-q,q}- \sum_{j=1}^{r}\frac{b_{r-j,j}}{b_{r-q,q}}\mu^\m E_{r-j,j}.
\end{equation*} We conclude that the parametric terms \(E_{r-q,q}\mu^\m\) and \(E_{s-p,p}\mu^\m\) for arbitrary nonzero \(\m\in \mathbb{Z}_{\geq0}^p\) are simplified in the $(r+1)$-th level partial parametric normal form. Eulerian terms \(a_{s+l-j, j}(\mu)E_{s+l-j, j}\) for \( 1\leq l\leq r-s-1\) and \(p\leq j\leq p+l\) are simplified in the \((s+1)\)-th level. Similar to the proof of Theorem \ref{r+1Orb}, we consider matrix representation \eqref{MatrxDr+1} and remove its third sub-matrix block. When \(l+r-s+1\leq q-p,\) the column spaces of matrices \(2(s-r)\mathcal{M}_r^l\) and \(\mathcal{M}_s^{l+r-s}\) are linearly independent and thus, \(E_{l+r-j, j}\)-terms for \( p\leq j\leq p+l+r-s\) and \(q\leq j\leq q+l\) are simplified while for \(p+l+r-s+1\leq j\leq q-1\) and \(0\leq j\leq p-1,\) \(E_{l+r-j, j}\)-terms may remain in the \(r+1\)-level normal form. Assume that \(q-p<l+r-s+1\). Similar to what is described in equations \eqref{RecurRelations}, \(E_{l+r-j, j}\)-terms for \(p\leq j\leq q-1\) can be simplified. Now we obtain the matrix
\begin{equation*}
\begin{bmatrix}
A_{l}&B_{l}\\
C_{l}&D_{l}
\end{bmatrix}
\end{equation*} by eliminating the first \(q\)-rows and \(q-p\)-columns starting from \(l+2\)-th column to \(l+q-p+1\)-th column of \([2(s-r) \mathcal{M}_{r}^{l}\quad \mathcal{M}_{s}^{l+r-s}]\); \ie we omit the first \(q-p\)-columns of \(\mathcal{M}_s^{l+r-s}\). By Lemma \ref{lemRank}, we have
%\( p \leq j \leq q+l+d_l \)
\begin{equation}
\dd_l=\left\{\begin{array}{lcc}
l+1                   & \text{if}    & q-p-r+s\leq l\leq \alpha-r-2,\\
[0.7ex] \alpha-r-1    &\text{when }  & l>\alpha-r-2.
\end{array}\right.
\end{equation} Hence all \(E_{l+r-j, j}\)-terms for \(p\leq j \leq p+l+r-s+\dd_l\) can be simplified.
\end{proof}

\begin{rem}
Eulerian parametric terms
\bes a_{s+l-j, j}E_{s+l-j, j}\mu^\m\quad\;\hbox{ for }\; 1\leq l\leq r-s-1\; \hbox{ and }\; p\leq j\leq p+l\ees are also simplified in the \((s+1)\)-th level parametric normal form. Now we complete the proof by a parametric version of Theorem \ref{r+1Orb} as follows. For \(l\leq p-q-1,\) Eulerian terms \(E_{l+r−j,j}\) for \(p\leq j\leq p+l+r-s\) and \(0\leq j\leq q-1\) are simplified while \(E_{l+r−j,j}\)-terms for \(q\leq j\leq p-1\) cannot be normalized in the $r+1$-th level normalization step. When $l+1\geq p-q,$ $E_{l+r−j,j}$-terms for $p\leq j\leq p+l+r-s+\dd_l$ are simplified while these cannot be normalized when \(p+l+r-s+\dd_l+1\leq j\leq l+r\) and \(\dd_l\) is defined by \eqref{ul1} and \eqref{ul2}.
\end{rem}
%%%%%%%%%%%%%%%%%%%%%%%%%%%%%%%%%%%%%%%%%%%%%%%%%%%%%%%%%%%%%%%%%%%%%%%%%%%%%%%%%%%%%%%%%%%%%%%%%%%%
\begin{cor} \label{cor 5.3}
Assume that  \( \rank \left[ \mathcal{M}_{r}^l \quad \mathcal{M}_{s}^{l+r-s} \right]=r+s+1 \) and \( p<q \). Then, the infinite level extended partial parametric normal form of \( w^{(1)} \) is given by
\begin{footnotesize}
\begin{eqnarray*}
&w^{(r+1)}=v_0+\sum_{i+j= 1}^{s}  \left( a^2_{i, j}(\mu)\Theta_{i, j}^2 + b_{i, j}(\mu)E_{i, j}\right)+\sum_{l=1, j=l+r-s+1, j=q+l+1}^{q-r+s-1, q-1, l+r}b_{l+r-j, j}(\mu)E_{l+r-j, j}\quad&\\
&+\sum_{l=p, j=2l+r-s+2}^{s-2, l+r}b_{l+r-j, j}(\mu)E_{l+r-j, j}+\sum_{l=1, j=0, j=l+p+1}^{r-s, p-1, l+s}b_{l+s-j, j}(\mu)E_{l+s-j, j}&\\
&+ \sum_{i+j=s} a^1_{i,j}(\mu)\Theta^{1}_{i, j}+\sum_{l=1, j=0, j=l+p+1}^{\infty, p-1, l+s} a^2_{l+s-j, j}(\mu)\Theta_{l+s-j, j}^2.\qquad\qquad\qquad\quad&
\end{eqnarray*}\end{footnotesize} Here for each \(1\leq l\leq r-s-1,\) \(b_{l+s-j, j}(\0)=0\) when \(0\leq j\leq p-1\) and \(p+l+1\leq j \leq l+s.\) Furthermore,  \(b_{s-p,p}(\mu)=b_{s-p,p}(0)\neq 0\) and \(b_{r-q,q}(\mu)=b_{r-q,q}(0)\neq 0.\)
%The vector field \(w^{(r+1)}\)  constitutes the infinite level orbital normal form of \(w^{(1)}.\)
\end{cor}
\begin{proof}
In this case, Proposition \ref{3.7Rem} implies \(p=0\). For \(1\leq l\leq q-r+s-1,\) Theorem \ref{r+1Orb} concludes that \(E_{l+r-j, j}\)-terms for \(0\leq j\leq l+r-s\) and \(q\leq j\leq q+l\) are simplified. Since \(\dd_l=l+1\) for \(q-r+s\leq l\leq s-1,\) \(E_{l+r-j, j}\)-terms for \(0\leq j\leq 2l+r-s+1\) are also normalized. However \(\dd_l=s\) for \(l>s-1\). Thus, all Eulerian terms of grade \(l+r\) can be normalized. Proof is complete by Theorem \ref{r+1PNFThm}.
\end{proof}
%%%%%%%%%%%%%%%%%%%%%%%%%%%%%%%%%%%%%%%%%%%%%%%%%%%%%%%%%%%%%%%%%%%%%%%%%
\begin{rem}\label{rem 5.5}
When \( p>q \) and \(\rank\left[\mathcal{M}_{r}^l\quad\mathcal{M}_{s}^{l+r-s}\right]=r+s+1\), parametric normal form follows a parametric version of Proposition \ref{cor4.4} and Theorem \ref{r+1PNFThm}. A similar argument to the case \(p>q\) in the \(r+1\)-level orbital normal form implies that the \(r+1\)-level normalization gives rise to an infinite level parametric normal form.
\end{rem}
%%%%%%%%%%%%%%%%%%%%%%%%%%%%%%%%%%%%%%%%%%%%%%%%%%%%%%%%%%%%%%%%%%%%%%%%%

\begin{thm}[The case \(r=\infty\)]\label{r=infty PNF}\begin{itemize}
\item[1.] Let \(r=\infty\) and \(s<\infty.\) Then, the infinite level parametric normal form is given by equation \eqref{s+1PNF} where \(E_{n-j, j}(0)=0\) for all \(s\neq n\geq 0\) and \(E_{s-p, p}(\mu)=E_{s-p, p}(0).\)
  \item[2.] Assume that \(r=\infty\) and \(s=\infty.\) Either the nonparametric part of the vector field is linearizable or there exists a
  natural number \(s_2\) so that its \(s_2+1\)-th level parametric normal form is
\begin{eqnarray}\label{s=infty, s2<infty PNF}
 \Scale[0.88]{ w^{(s_2+1)}:=v_0\!+\!\sum_{i+j=0}^{\infty}b_{i, j}(\mu)E_{i, j}\!+\!\sum_{i+j=1}^{s_2}a^2_{i, j}(\mu)\Theta^2_{i, j}
 \!+\!\sum_{l=1, j=0, l+p_2+1}^{\infty, p_2-1, l+s_2}a^2_{l+s_2-j, j}(\mu)\Theta^2_{l+s_2-j, j}}
  \end{eqnarray} where \(a^2_{s_2-p_2, p_2}=a^2_{s_2-p_2, p_2}(0)\neq 0.\) The vector field \( w^{(s_2+1)}\) is the infinite level parametric normal form.
  %\item[4.] If \(s=\infty,s_2=\infty\), the equation \eqref{first level PNF s=infty} is infinite level parametric normal form of \( w^{(1)}\) where \(a^2_{i, j}(0)=0\) for \(i+j\geq 1\) and \(b_{i, j}(0)=0\) for all \(i+j\geq 0\).
  \end{itemize}
\end{thm}
\begin{proof} By a parametric version of the proof in Theorem \ref{s+1PNFThm}, equation \eqref{s+1PNF} is the infinite level parametric normal form.
Assuming that \(r=\infty\) and \(s=\infty,\) we define the grading function \(\delta(\mu^{\m}\Theta^i_{n-j, j})=|\m|+n+i-1\) for \(i=1, 2\) and \(\delta(\mu^{\m}E_{n-j, j})=|\m|+n\) for the first level parametric normal form \(v^{(1)}.\) Via a parametric version of Theorem \ref{r=infty ONF thm}, the first level reads
\begin{eqnarray}\label{first level PNF s=infty}
& w^{(1)}:=v_0+\sum_{i+j=0}^{\infty}b_{i, j}(\mu)E_{i, j}+\sum_{i+j=1}^{\infty}a^2_{i, j}(\mu)\Theta^2_{i, j}, \quad
b_{i, j}(0)=0.&
\end{eqnarray} When \(s=\infty\) and \(s_2=\infty,\) the nonparametric part of the vector field is linearizable. Otherwise, let \(s_2<\infty.\)
For this case, the grading function for \(s_2+1\)-level is updated by \(\delta(\mu^{\m}\Theta^i_{n-j, j})=|\m|(s_2+1)+n\) for \(i=1, 2\) and \(\delta(\mu^{\m}E_{n-j, j})=|\m|(s_2+1)+n\). Now by \(\left[\mu^{\m}E_{0, 0}, \Theta^2_{s_2-j, j}\right]= -2s_2\mu^{\m}\Theta^2_{s_2-j, j}\) and employing a parametric version of item 2 in Theorem \ref{r=infty ONF thm}, the claim is obtained.
  %Item 4. Take \( \bar{s}_2:= \min \{|\m|+n|a^2_{n-j, j,\m}\neq 0\}\). Define the grading function \(\delta(\mu^{\m}\Theta^2_{n-j, j})=|\m|(\bar{s}_2+1)+n+i-1\) for \(i=1, 2\) and \(\delta(\mu^{\m}E_{n-j, j})=|\m|(\bar{s}_2+1)+n\)
\end{proof}

%%%%%%%%%%%%%%%%%%%%%%%%%%%%%%%%%%%%%%%%%%%%%%%%%

\begin{thm}\label{inftPar}
Assume that %$\left(b_{1,0}(0)+b_{0, 1}(0)\right)b^{(2)}_{0,2}(0)=\sum_{i=0}^{1}b_{1, 1-i}(0)b^{(2)}_{0,2}(0)\neq 0$.
${b_{0,1}(0)}^2 b_{2,0}(0)-b_{0,1}(0)b_{1,0}(0)b_{1,1}(0)+b_{0,2}(0){b_{1,0}(0)}^2\neq 0$. When \(b_{1,0}(0)\neq 0,\) the infinite level parametric normal form of \(v\) is given by
%\bas
%&v^{(2)}:= v_0+\sum_{i+j=2} a^2_{i,j}(\mu)\Theta^{2}_{i, j}+\sum_{i+j= 1}\left( a^1_{i,j}(\mu)\Theta_{i,j}^1+b_{i,j}(\mu)E_{i,j}\right) + \sum^\infty_{i=2} \left( b_{0,j}(\mu)E_{0,j}+ a^1_{0,j}(\mu)\Theta_{0,j}^1 \right).&
%\eas Further assume that $ b_{0,2}(0)=b_{0,2}\neq 0 $ in the second level parametric normal form. Then, the third level parametric normal form of
%\(v^{(1)}\) is given by
\ba\label{InfinitePNFS1R2Case1}
\Scale[0.88]{ v^{(\infty)}:=v_0\!+\!\sum_{i+j=1}a^1_{i,j}(\mu)\Theta_{i,j}^1\!+\!\sum_{i+j=0}^1\left( a^2_{i,j}(\mu)\Theta_{i,j}^2\!+\!b_{i,j}(\mu)E_{i,j}\right)\!+\!b^{(2)}_{0, 2}(0)E_{0,2}\!+\!\sum^\infty_{j=2} a^2_{0,j}(\mu)\Theta_{0,j}^{{2(2)}}}
\ea
and \(b^{(2)}_{0,2}(0)=\frac{{b_{01}}^2b_{20}-b_{01}b_{10}b_{11} + {b_{10}}^2b_{02}}{{b_{10}}^2}.\) Here \(b_{ij}\) denotes \(b_{i, j}(0).\) For \(b_{10}=0, b_{01}\neq 0\), the infinite level parametric normal form of \(v\) is expressed as
%\ba\nonumber
%&v^{(\infty)}:=v_0+\sum_{i+j=1}a^1_{i,j}(\mu)\Theta_{i,j}^1+\sum_{i+j=0}^1\left( a^2_{i,j}(\mu)\Theta_{i,j}^2+b_{i,j}(\mu)E_{i,j}\right) +b^{(2)}_{2, 0}(0)E_{2, 0}&\\\label{InfinitePNFS1R2Case2}&
%+ \sum^\infty_{i=2} a^2_{j, 0}(\mu)\Theta_{j, 0}^{2(2)},\qquad\qquad\qquad\qquad\qquad\qquad\qquad\qquad&
%\ea
\ba\label{InfinitePNFS1R2Case2}
\Scale[0.88]{v^{(\infty)}:=v_0\!+\!\sum_{i+j=1}\!a^1_{i,j}(\mu)\Theta_{i,j}^1\!+\!\sum_{i+j=0}^1\!\left( a^2_{i,j}(\mu)\Theta_{i,j}^2\!+\!b_{i,j}(\mu)E_{i,j}\right)\!+\!b^{(2)}_{2, 0}(0)E_{2, 0}\!+\!\sum^\infty_{j=2} a^2_{j, 0}(\mu)\Theta_{j, 0}^{2(2)}}
\ea
and \(b^{(2)}_{2, 0}(0)={b_{01}}^{-2}({b_{01}}^2b_{20}-b_{01}b_{10}b_{11} + {b_{10}}^2b_{02}).\)
\end{thm}\!
\bpr Since \({b_{0,1}(0)}^2 b_{2,0}(0)-b_{0,1}(0)b_{1,0}(0)b_{1,1}(0)+b_{0,2}(0){b_{1,0}(0)}^2\neq 0\),
\bas&\sum_{j=0}^{1}{b_{j, 1-j}(0)}^2\neq 0, \quad\sum_{j=0}^{2}{b_{j, 2-j}(0)}^2\neq 0, \quad \text{and} \quad s=1.&\eas When \(b_{1, 0}(0)\neq 0,\) \(p=0\). Consider the normal form coefficients \(b_{20\m}\) and \(b_{11\m}\in\mathbb{R}.\) Hence by Theorem \ref{s+1PNFThm}, the second level parametric normal form is
%\bas
%&v^{(2)}:=v_0+\sum_{i+j=0}^1\left( a^2_{i,j}(\mu)\Theta_{i,j}^2+b_{i,j}(\mu)E_{i,j}\right) +\sum_{i+j=1} a^{1}_{i,j}(\mu)\Theta^{1}_{i, j}&\\&
%+ \sum^\infty_{i=2} \left(a^{2(2)}_{0,j}(\mu)\Theta_{0,j}^2+b_{0,j}^{(2)}(\mu)E_{0,2}\right)\qquad\qquad&
%\eas
\[\Scale[0.9]{v^{(2)}:=v_0+\sum_{i+j=0}^1\left( a^2_{i,j}(\mu)\Theta_{i,j}^2+b_{i,j}(\mu)E_{i,j}\right) +\sum_{i+j=1} a^{1}_{i,j}(\mu)\Theta^{1}_{i, j}
+ \sum^\infty_{i=2} \left(a^{2(2)}_{0,j}(\mu)\Theta_{0,j}^2+b_{0,j}^{(2)}(\mu)E_{0,2}\right)}\]

 and \(b_{0, 2}^{(2)}(0):=b_{0, 2}(0)+\frac{{b_{01}}^2b_{20}-b_{01}b_{10}b_{11}}{{b_{10}}^2}.\)
%={b_{10}}^{-2}\left({b_{0,1}}^2 b_{2,0}-b_{0,1}b_{1,0}b_{1,1}+b_{0,2}{b_{1,0}}^2\right)\neq 0.$$
Hence, \(r=2,\) \(q=2\) and \(q>p\). Theorem \ref{r+1PNFThm} and the equation
$d^{l+3|\m|+2,3}(\frac{b_{0l+2}}{2 b_{02}}(-2(l-1)Z_{0,l},0, E_{0,l},0)\mu^\m)=-b_{0l+2}E_{0,l+2}\mu^\m$ for \(\m\in \mathbb{Z}_{\geq0}^N\) where
\bas
&\sum^l_{j=0}\mathbb{R}(-2(l-1)Z_{l-j,j},0, E_{l-j,j},0)\mu^\m\subseteq \ker d^{l+3|\m|+1, 2},&
\eas imply that the third level parametric normal form of \(v^{(2)}\) is given by equation \eqref{InfinitePNFS1R2Case1}. Hence,  \(\rank[\mathcal{M}^2_3\quad\mathcal{M}^3_2]=4\) by Proposition \ref{3.7Rem}. Proposition \ref{cor4.4} concludes that the normal form vector field \(v^{(3)}\) is the infinite level parametric normal form. The block matrices \(A_l,\) \(B_l,\) \(C_l\) and \(D_l\) are obtained by removing the first two rows of $\left[-2\mathcal{M}_{2}^{l} \quad \mathcal{M}_{1}^{l+1} \right]$ and two columns (\(l + 2\)-th and \(l + 3\)-th) from $\left[-2\mathcal{M}_{2}^{l} \quad \mathcal{M}_{1}^{l+1} \right].$ Then,
\begin{eqnarray*}
\begin{bmatrix}
B_l \\
D_l
 \end{bmatrix}=\begin{bmatrix}
           b_{1, 0}  & b_{0, 1}  &  0            & \cdots          &0\\
            0        & b_{1, 0}  &  \ddots       &\ddots           &\vdots\\
           \vdots    & \ddots    & \ddots        &\ddots           &0\\
           \vdots    & \ddots    &\ddots         & b_{1, 0}        &b_{0, 1}\\
            0        & \cdots    &\cdots         &0                & b_{1, 0}
                 \end{bmatrix}^t, %color blue
&
\end{eqnarray*}
\(B_l= b_{1, 0}I_{l\times l}+b_{0, 1}
\begin{bmatrix}
\0_{1\times (l-1)} & 0\\
I & \0
\end{bmatrix},\)
\(A_l=-2\left[b_{0, 2}I_{l\times l}\,\,\0_{l\times 1}\right],\) \(C_l=-2\left[\0_{1\times l}\,\,b_{0, 2} \right],\) and \(D_l=\left[\0_{1\times (l-1)}\,\,b_{0, 1} \right].\) By Proposition \ref{cor4.4}, \(\dd_l=1\) for all \(l\geq 1\).
For \(b_{1,0}=0\) and \(b_{0, 1}\neq 0,\) \(p=1\). By Theorem \ref{s+1PNFThm}, the second level parametric normal form is read by
%\bas
%&v^{(2)}:=v_0+\sum_{i+j=0}^1\left( a^2_{i,j}(\mu)\Theta_{i,j}^2+b_{i,j}(\mu)E_{i,j}\right) +\sum_{i+j=1} a^{1}_{i,j}(\mu)\Theta^{1}_{i, j}&\\&
%+ \sum^\infty_{i=2} \left(a^{2(2)}_{j, 0}(\mu)\Theta_{j, 0}^2+b_{j, 0}^{(2)}(\mu)E_{2, 0}\right).\qquad\qquad&
%\eas
\[\Scale[0.9]{v^{(2)}:=v_0+\sum_{i+j=0}^1\left( a^2_{i,j}(\mu)\Theta_{i,j}^2+b_{i,j}(\mu)E_{i,j}\right) +\sum_{i+j=1} a^{1}_{i,j}(\mu)\Theta^{1}_{i, j}
+ \sum^\infty_{j=2} \left(a^{2(2)}_{j, 0}(\mu)\Theta_{j, 0}^2+b_{j, 0}^{(2)}(\mu)E_{2, 0}\right)}.\]
 By \({b_{0,1}(0)}^2 b_{2,0}(0)-b_{0,1}(0)b_{1,0}(0)b_{1,1}(0)+b_{0,2}(0){b_{1,0}(0)}^2\neq 0\), \(b_{2, 0}^{(2)}(0)=b_{2, 0}(0)\neq 0,\) \(r=2\) and \(q=0\). Now by parametric version of Theorem \ref{s1r2} and Theorem \ref{r+1PNFThm}, the third level parametric normal form of \(v^{(2)}\) is given in equation \eqref{InfinitePNFS1R2Case2}. By Remark \ref{rem 5.5}, the vector field \eqref{InfinitePNFS1R2Case2} is an infinite level parametric normal form.
\epr

\section{First level normal form coefficients }\label{Sec6}

%The main computer-computational burden is due to the cost for the first level normal form derivation. Hence,
A new and efficient algorithm is here proposed to derive the first level truncated normal form formulas for nonlinear singular
Eulerian vector fields given by
\begin{eqnarray}\label{Coefficients Vector field}
v(\z):= \omega_1\Theta_{0, 0}^1+ \omega_2\Theta_{0, 0}^2+ E_{f},\; \hbox{ for } f\in \mathbb{R}[[\z]], \z:= (z_1, w_1, z_2, w_2)\in \mathbb{C}^4 \hbox{ and } w_i=\bar{z}_i.
\end{eqnarray} All even-degree homogeneous vector fields are eliminated in the first level normal form. Thus, we always have \(f^{l}_{2k}:= \sum_{j=1}^{k} b_{k-j,j} |z_1|^{2(k-j)}|z_2|^{2j}\) for \(l\geq 2k,\) and
\(f^{l}_{2k+1}:= 0\) for \(l\geq 2k+1.\) Denote the transformation generator for simplification of grade-\(k\) homogenous part of \(v^{k-1}\) by \(E_{h_{k}},\) that is determined by
\ba\label{h_k}
&h_k:=\Scale[0.9]{\sum_{\substack{i_1+i_2+j_1+j_2=k,\\ i_1\neq j_1, i_2\neq j_2}}\left(\dfrac{I(i_1!j_1!i_2!j_2!)^{-1}}{(i_1-j_1)\omega_1+(i_2-j_2)\omega_2}\dfrac{\partial^k\,  f^{k-1}(\z)}{\partial {z_1}^{i_1} \partial {w_1}^{j_1}\partial {z_2}^{i_2}\partial {w_2}^{j_2}}\big\vert_{z=\0}\right){z_1}^{i_1} {w_1}^{j_1} {z_2}^{i_2} {w_2}^{j_2}.}\quad&
\ea

\begin{thm}\label{Coeifficient Thm}
The first level normal form of vector field \eqref{Coefficients Vector field}   is given by
\bas
&v^{(1)}:= \omega_{1}\Theta^1_{0,0}+\omega_{2}\Theta_{0,0}^2 + \sum_{k=1}^{\infty}\sum_{j=1}^{k} b_{k-j,j}E_{k-j,j}.&
\eas where \(b_{k-j,j}\)-s are given by
\ba\label{CoefficentOfNormalForm}
&b_{k-j,j}:= \dfrac{\partial^{2k}\, f^{2k-1}}{(2(k-j))!(2j)!\partial |z_1|^{2(k-j)} \partial |z_2|^{2j}}\Big\vert_{\z=\0},&
\ea while the \(n\)-jet truncation of \(f^k\) recursively follows the equation
\ba\label{EulerianRecursiveRelation}
&J^nf^{k}= J^n f^{k-1}+\left< \nabla h_{k},v_0 \right> +\sum_{i=1}^{n-k}\sum_{m=1}^{\lfloor \frac{n-i}{k} \rfloor}\,\dfrac{1}{m!}\prod_{j=2}^{m+1}\left( (j-m)k-i\right) {h_{k}}^{m}f_{i}^{k-1}, \quad&
\ea for \(n\geq k,\) and \(f^0:=f.\)
\end{thm}
\begin{proof}
Since
\begin{eqnarray*}
&[{z_1}^{m_1}{w_1}^{n_1}{z_2}^{m_2}{w_2}^{n_2} E_{0, 0}, v_0]=\big(\omega_1({m_1}-{n_1})+\omega_2({m_2}-{n_2})\big) {z_1}^{m_1}{w_1}^{n_1}{z_2}^{m_2}{w_2}^{n_2} E_{0, 0},&\\
&\big[\frac{{z_1}^{m_1}{w_1}^{n_1}{z_2}^{m_2}{w_2}^{n_2}}{{m_1}+{n_1}+{m_2}+{n_2}-{p_1}-{q_1}-{p_2}-{q_2}} E_{0, 0}, {z_1}^{{p_1}}{w_1}^{{q_1}}{z_2}^{{p_2}}{w_2}^{{q_2}}E_{0, 0}\big]={z_1}^{{m_1}+{p_1}}{w_1}^{{n_1}+{q_1}}{z_2}^{{m_2}+{p_2}}{w_2}^{{n_2}+{q_2}} E_{0, 0},&
\end{eqnarray*} all monomial vector fields with odd degrees can be eliminated from the first level normal form. Further, for the case of homogenous odd-degree vector fields, we can eliminate all terms except \(|z_1|^{2i_1} |z_2|^{2i_2} E_{0,0}\) for \(i_1,\) \(i_2\in \mathbb{N}.\)
For an even number \(k,\) the \(k+1\)-degree homogeneous vector field part of \(v^{k-1}\) follows
\bas
&\sum_{2i_1+2i_2=k}\left(\dfrac{\partial^k\,  f^{k-1}}{(2i_1)!(2i_2)!\partial {|z_1|}^{2i_1} \partial {|z_2|}^{2i_2}}\Big\vert_{z=\0} \right)|z_1|^{2i_1} |z_2|^{2i_2} E_{0,0}.&
\eas The transformation generator $ E_{h_{k}} $ gives rise to
\begin{eqnarray*}
&v^k:=\exp \ad_{E_{h_k}} v^{k-1} = v^{k-1}+ \sum_{m=1}^{\infty} \frac{1}{m!} \ad^{m}_{E_{h_k}} v^{k-1}= v^{k-1}+ \sum_{m=1}^{\infty} \sum^\infty_{i=0} \frac{1}{m!} \ad^{m}_{E_{h_k}} v^{k-1}_{i},&
\end{eqnarray*} and \(v_i^{k-1}\in \LST_i.\) Since $v^{k-1}_{0}=v_0,$ $v^{k-1}_{i}=E_{ f^{k-1}_{i}},$ and $ v^{k-1}:= v_0 + E_{ f^{k-1}},$
\begin{eqnarray*}
&v^{k}= v^{k-1}+\ad_{E_{h_k}}v_0 +\sum_{m=1}^{\infty}\sum_{i=1}^{\infty}\,\dfrac{1}{m!}\prod_{j=2}^{m+1}\left( (j-m)k-i\right) {h_{k}}^{m}f_{i}^{k-1} E_{0,0}&
\end{eqnarray*} and equation \eqref{EulerianRecursiveRelation} holds.
\end{proof}

%\begin{proc}\label{NFCAlg} Computation of the first level normal form of a vector field type  \eqref{Coefficients Vector field}.
%\begin{itemize}
%  \item[]{\bf Input: } Vector field \(v\) type \eqref{Coefficients Vector field}.\vspace{-0.3cm}
%  \item[]{\bf Output:} \(v^{(1)}\) which is the first level normal form of the vector field \(v\) up to grade \(n\).
%\end{itemize}
%\begin{enumerate}
%\item[(1)] Let \(f^0\) be the \(n\) degree truncated of the scalar function of the Eulerian term of \(v\) .
%\item[(2)] Take \(v^{(1)}:=\omega_1\Theta^1_{0,0}+\omega_2\Theta^2_{0, 0}\) and \(k:=1.\)
%\item[(3)] Do the following as long as \(k\leq 2n-1.\)
%\begin{enumerate}
%\item Take \(f^{k-1}=\sum_{k=1}^{n}f^{k-1}_j\) where \(f^{k-1}_j\)'s  are the homogeneous term of degree \(j\) of \(f^{k-1}\) for \(1\leq j\leq k-1\).
%\item Define the transformation generator \(h_k\)  according to equation \eqref{h_k}.
%\item Use the equation \ref{EulerianRecursiveRelation} to simplify \(f^{k-1}_{k}\) and update \(J^nf^{k}.\)
%\item If \(k\) is odd, while  \(1\leq i\leq \lfloor \frac{k}{2}\rfloor+1\) do
%\begin{enumerate}
%  \item Compute \(b_{k-i, i}\)  from equation \eqref{CoefficentOfNormalForm},
%  \item Let \(v^{(1)}:=v^{(1)}+b_{k-i, i}|z_1|^{2(k-i)}|z_2|^{2i}{E}_{0, 0}.\)
%
%  end while, end if.
%\end{enumerate}
%\item Take \(k:=k+1.\)
%\item Check the condition \(k\leq 2n-1\).
%\end{enumerate}
%\item[(3)] Return \(v^{(1)}\).
%\end{enumerate}
%\end{proc}

Theorem \ref{Coeifficient Thm} introduces the following algorithm for computing a truncated first level normal form of the vector field \eqref{Coefficients Vector field}.

\begin{algorithm}[H]
\caption{Computation of a truncated first level normal form. } % give the algorithm a caption
\label{alg1}
\begin{algorithmic}
\STATE {\bf Inputs:} \((v, n)\): Vector field \(v\) given by \eqref{Coefficients Vector field} and a natural number \(n.\)
\STATE {\bf Output:} \(J^{2n}v^{(1)}\): A \(2n\)-grade truncation of the first level normal form \(v^{(1)}.\) \\

    \STATE Let \(J^{2n}f^0\) be the \(2n\)-degree truncation of the scalar function \(f\) associated with \(v\).
     \STATE  Take \(\hat{v}^{(1)}:=\omega_1\Theta^1_{0,0}+\omega_2\Theta^2_{0, 0}\) and \(k:=1.\)
     \WHILE{$k\leq 2n-1$}
      \STATE Take \(J^{2n}f^{k-1}=\sum_{j=1}^{2n}f^{k-1}_j\) where \(f^{k-1}_j\) is the \(j\)-degree homogeneous polynomial part of \(f^{k-1}\) for
      \(1\leq j\leq 2n\).
      \STATE Define the transformation generator \(h_k\)  according to equation \eqref{h_k}.
        \STATE Use the equation \ref{EulerianRecursiveRelation} to simplify \(f^{k-1}_{k}\) and update \(J^{2n}f^{k-1}\) with \(J^{2n}f^{k}.\)
        \IF{$k$ is odd }
        \WHILE{$1\leq i\leq \lfloor \frac{k}{2}\rfloor+1$}
        \STATE Compute \(b_{k-i, i}\)  from equation \eqref{CoefficentOfNormalForm}.
        \STATE Let \(\hat{v}^{(1)}:=\hat{v}^{(1)}+b_{k-i, i}|z_1|^{2(k-i)}|z_2|^{2i}{E}_{0, 0}.\)
        \ENDWHILE
        \ENDIF
       \STATE Take \(k:=k+1.\)
        %\STATE Check the condition \(k\leq 2n-1\).
        \ENDWHILE
        \STATE Set \(J^{2n}v^{(1)}:= \hat{v}^{(1)}.\)
       \RETURN \(J^{2n}v^{(1)}\).
\end{algorithmic}
\end{algorithm}

\begin{cor}\label{Cor6.2}
Consider the vector field \eqref{Coefficients Vector field} where
\begin{footnotesize}\begin{eqnarray*}
&f(\x):= a_1 x_1+a_2y_1 +a_3 x_2+a_4 y_2+ a_{5}{x_1}^{2}+a_{6}x_1y_1+a_{7}x_1x_2+a_{8}x_1y_2+a_{9}{y_1}^{2}+a_{10}y_1x_2+a_{11}y_1y_2&\\
&+a_{12}{x_2}^{2}+a_{13}x_2y_2+a_{14}{y_2}^{2}.\qquad\qquad\qquad\qquad\qquad\qquad\qquad\qquad\qquad\quad&
\end{eqnarray*}\end{footnotesize} The vector field \eqref{Coefficients Vector field} can be transformed to the \(6\)-jet truncated normal form
\begin{footnotesize}
\bas
&v^6(\x):= \omega_1\Theta_{0, 0}^1+ \omega_2\Theta_{0, 0}^2+ \sum_{k=1}^{3}\sum_{j=0}^{k}b_{k-j,j}E_{k-j,j}&
\eas where $b_{2,1}, b_{1,2}$ are given in appendix and Theorem \ref{Coeifficient Thm} and
\ba\nonumber
&b_{1,0} = \frac{a_{5}+ a_{9}}{2}, \,\, b_{0,1} = \frac{a_{12}+a_{14}}{2}, \,\, b_{1,1} = \frac{a_{{11}}a_{{1}}a_{{3}}-a_{{10}}a_{{4}}a_{{1}}-a_{{8}}a_{{3}}a_{{2}}+a_{{7}}a_{{4}}a_{{2}}}{2\,{\omega_1}{\omega_2}}\!+\!\frac{ {a_{{3}}}^{2}+{a_{{4}}}^{2} }{4\,{\omega_2}^2( a_{{5}}+a_{{9}})^{-1}}\!+\!\frac{ {a_{{1}}}^{2}+{a_{{2}}}^{2} }{4\,{\omega_1}^2\left(a_{{12}}+a_{{14}}\right)^{-1}}, &\\\nonumber
&b_{2,0} = \frac{a_{5}{a_{1}}^{2}+3\,a_{5}{a_{2}}^{2}-2\,a_{{6}}a_{1}a_{2}+3\,a_{9}{a_{1}}^{2}+\,a_{9}{a_{2}}^{2}}{{8\,\omega_1}^2}, \quad b_{0,2} = \frac{a_{{12}}{a_{{3}}}^{2}+3\,a_{{12}}{a_{{4}}}^{2}-2\,a_{{13}}a_{{3}}a_{{4}}+3\,a_{{14}}{a_{{3}}}^{2}+\,a_{{14}}{a_{{4}}}^{2}}{{8\omega_2}^2},\qquad\quad\, &\\\label{biCoeff}
&b_{3,0} = \frac{\left( a_{5}+a_{9} \right)  \left( {a_{1}}^{2}a_{6}-2a_{1}a_{2}a_{5}+2a_{1}a_{2}a_{9}-{a_{2}}^{2}a_{6} \right)}{16{\omega_1}^3}+\frac{\left( {a_{1}}^{2}+{a_{2}}^{2} \right)  \left( {a_{1}}^{2}a_{5}+5{a_{1}}^{2}a_{9}-4a_{1}a_{2}a_{6}+5{a_{2}}^{2}a_{5}+{a_{2}}^{2}a_{9} \right)}{16{\omega_1}^4},\qquad\quad\,&
\\\nonumber
&b_{0,3} =\frac{\left( a_{12}+a_{14} \right)  \left( {a_{3}}^{2}a_{13}-2a_{3}a_{4}a_{12}+2a_{3}a_{4}a_{14}-{a_{4}}^{2}a_{13}
\right)}{{64\omega_2}^3} +\frac{  \left( {a_{3}}^{2}a_{12}+5{a_{3}}^{2}a_{14}-4a_{3}a_{4}a_{13}+5{a_{4}}^{
2}a_{12}+{a_{4}}^{2}a_{14}\right)}{{64\omega_2}^4\left({a_{3}}^{2}+{a_{4}}^{2} \right)^{-1}}.\qquad\quad\, &
\ea \end{footnotesize}
\end{cor}
\begin{proof}
Using the changes of coordinates to the complex coordinates,
we have
\begin{eqnarray*}
&f_{0}:=f(\z)=g\left(\frac{z_1+w_1}{2},\frac{z_1-w_1}{2I},\frac{z_2+w_2}{2},\frac{z_2-w_2}{2I}\right)= f^1_0(\z)+f^2_0(\z),&
\end{eqnarray*} where \(f^1_0(\z):=  A_1 z_1+A_2 w_1 +A_3 z_2+A_4 w_2,\)
\begin{eqnarray*}
&\Scale[0.9]{f^2_0(\z)\!:=\! A_{5}{z_1}^{2}\!+\!A_{6}z_1w_1\!+\!A_{7}z_1z_2\!+\!A_{8}z_1w_2\!+\!A_{9}{w_1}^{2}\!+\!A_{10}w_1z_2\!+\!A_{11}w_1w_2 \!+\!A_{12}{z_2}^{2}\!+\!A_{13}z_2w_2\!+\!A_{14}{w_2}^{2}.}&
\end{eqnarray*} Recall $h_1(\z)$ from \eqref{CoefficentOfNormalForm} by \(h_1(\z) :=\frac{IA_1}{\omega_1}  z_1+\frac{IA_2}{-\omega_1} w_1 +\frac{IA_3}{\omega_2} z_2+\frac{IA_4}{-\omega_2} w_2.\) By applying this transformation generator and Theorem \eqref{Coeifficient Thm}, $\left< \nabla h_1,\Theta \right>=-f^0_1.$ Now we have \(b_{1,0}= \frac{\partial f^1}{\partial |z_1|^{2}}\big|_{z=0}=A_{5}= \frac{1}{2}a_{5}+\frac{1}{2}a_{9},\) \(b_{0,1}=\frac{\partial f^1}{\partial |z_2|^{2}}|_{z=0}=A_{12}= \frac{1}{2}a_{12}+\frac{1}{2}a_{14}\) and
\begin{eqnarray*}
&\Scale[0.95]{h_2(\z)\!:=\! \frac{IA_5}{2\omega_1}{z_1}^{2}\!+\! \frac{IA_7}{\omega_1+\omega_2}z_1z_2\!+\! \frac{IA_8}{\omega_1-\omega_2}z_1w_2\!-\! \frac{IA_9}{2\omega_1}{w_1}^{2}\!+\!\frac{IA_{10}}{\omega_2-\omega_1}w_1z_2
\!-\! \frac{IA_{11}}{\omega_1+\omega_2}w_1w_2\!+\! \frac{IA_{12}}{2\omega_2}{z_2}^{2}\!+\!\frac{IA_{14}}{-2\omega_2}{w_2}^{2}.}&
\end{eqnarray*} Employing \(h_2(\z),\) we obtain
\begin{eqnarray*}
&f^2 = f^1+ \left< \nabla h_2,\Theta \right> + \sum_{i=1}^{4}\sum_{m=1}^{\lfloor \frac{ 6-i}{2}\rfloor} \dfrac{1}{m!}\prod_{j=2}^{m+1}\left( (j-m)k-i\right) {h_{2}}^{m}f_{i}^{1},&\\
&f^2 = b_{1,0}|z_1|^2+b_{0,1}|z_2|^2-h_1 f^0_2 + h_1 h_2 f^0_2 +{h_1}^2 f^0_2 -2 {h_1}^2 h_{2} f^0_2 - {h_1}^3 f^0_2 + {h_1}^4 f^0_2.&
\end{eqnarray*} Similarly, \(f^i\) for \(i= 3, 4, 5\) is obtained. Then, the \(6\)-jet truncation of the first level normal form can be extracted from the following formulas:
\bas
&J^6f^3-J^2f^2 = {h_1}^2 f^0_2 + h_3(b_{1,0}|z_1|^2+b_{0,1}|z_2|^2)+ h_1 h_2 f^0_2  - {h_1}^3 f^0_2 + {h_1}^4 f^0_2-2 {h_1}^2 h_{2} f^0_2,&
\\&J^6f^4 -J^3f^3= \sum_{i+j=2}b_{i,j}|z_1|^{2i}|z_2|^{2j}+(h_3+2h_4)\sum_{i+j=1}b_{i,j}|z_1|^{2i}|z_2|^{2j}+ h_1 h_2 f^0_2&\\&
- {h_1}^3 f^0_2 -2 {h_1}^2 h_{2} f^0_2+ {h_1}^4 f^0_2,\qquad\quad &\\
&\Scale[0.96]{J^6f^5\!-\!J^4f^4\!=\!{h_1}^4 f^0_2\!-\!2{h_1}^2 h_{2} f^0_2\!+\!2h_4\sum_{i+j=1}\!b_{i,j}|z_1|^{2i}|z_2|^{2j},\,J^6f^6\!-\!J^5f^5\!=\! \sum_{i+j=3}\!b_{i,j}|z_1|^{2i}|z_2|^{2j},}&
\eas
where
\begin{eqnarray*}
&b_{i,j}=\left\{ \begin{array}{lcc}
\dfrac{\partial^{2i+2j}}{(2i)!(2j)!\partial^{2i} |z_1|^{2i}\partial^{2j}|z_2|^{2j}}\, f_2^0 & \text{for}& i+j=1\\[0.7ex]
\dfrac{\partial^{2i+2j}}{(2i)!(2j)!\partial^{2i} |z_1|^{2i}\partial^{2j}|z_2|^{2j}}\,{h_1}^2\,f^{0}_{2}&\text{for}& i+j=2,
\end{array} \right.&
\end{eqnarray*} and \(b_{i,j}=\frac{\partial^{2i+2j}}{(2i)!(2j)!\partial^{2i} |z_1|^{2i}\partial^{2j}|z_2|^{2j}}\,\left( {h_1}^4\,f^{0}_{2}-2  h_2{h_1}^2\,f^{0}_{2}+2h_4\sum_{i+j=1}b_{i,j}|z_1|^{2i}|z_2|^{2j}\right)\) for \(i+j=3.\) This gives rise to the normal form coefficients \eqref{biCoeff} and \(b_{2,1},\) \(b_{1,2}\) in the appendix.
\end{proof}
\section{Appendix}

The following normal form coefficients are associated with Corollary \ref{Cor6.2}:
\begin{footnotesize}
\begin{eqnarray*}
&b_{2,1}:=\frac{3( {a_{{1}}}^{2}(a_{{5}}+3a_{{9}})-2a_{{1}}a_{{2}}a_{{6}}+{a_{{2}}}^{2}(3a_{{5}}+a_{{9}}) )}{8{\omega_2}^2({\omega_1}^2-{\omega_2}^2)({a_{{3}}}^{2}+{a_{{4}}}^{2} )^{-1}}
+ \frac{ ( a_{{1}}a_{{3}}a_{{11}}-a_{{1}}a_{{4}}a_{{10}}-a_{{2}}a_{{3}}a_{{8}}+a_{{2}}a_{{4}}a_{{7}} )}{4{\omega_1}{\omega_2}({\omega_1}^2-{\omega_2}^2)( {a_{{1}}}^{2}+{a_{{2}}}^{2} )^{-1}}-\frac{3( {a_{{1}}}^{2}+{a_{{2}}}^{2} )^{2} ( a_{{12}}+a_{{14}} )}{16{\omega_1}^{4}{\omega_2}^{-2}({\omega_1}^2-{\omega_2}^2)}&\\
 &+\frac{3(a_{{12}}+a_{{14}})({a_{{1}}}^{2}+{a_{{2}}}^{2})^{2}-6({a_{{3}}}^{2} +{a_{{4}}}^{2})({a_{{1}}}^{2}a_{{5}}+3{a_{{1}}}^{2}a_{{9}}-2a_{{1}}a_{{2}}a_{{6}}+3{a_{{2}}}^{2}a_{{5}}+{a_{{2}}}^{2}a_{{9}})}{16{\omega_1}^{2}({\omega_1}^2-{\omega_2}^2)}
 -\frac{(a_{{1}}a_{{3}}a_{{11}}-a_{{1}}a_{{4}}a_{{10}}-a_2a_3a_8+a_2a_{{4}}a_{{7}})}
 {4{\omega_1}^{3}{\omega_2}^{-1}({\omega_1}^2-{\omega_2}^2)({a_{{1}}}^{2}+{a_{{2}}}^{2})^{-1}}&\\
& +\frac{(a_{{3}}a_{{8}}-a_{{4}}a_{{7}} ) ( 3a_{{1}}a_{{5}}+5a_{{1}}a_{{9}}-a_{{2}}a_{{6}} ) - ( a_{{3}}a_{{11}}-a_{{4}}a_{{10}} )  ( a_{{1}}a_{{6}}-5a_{{2}}a_{{5}}-3a_{{2}}a_{{9}} ) }{8{\omega_2}{\omega_1}^{3}({\omega_1}^2-{\omega_2}^2)}-\frac{ ( a_{{12}}+a_{{14}} )( {a_{{1}}}^{2}a_{{6}}-2a_{{1}}a_{{2}}a_{{5}}+2a_{{1}}a_{{2}}a_{{9}}-{a_{{2}}}^{2}a_{{6}})}{16{\omega_1}^{3}{\omega_2}^{-2}({\omega_1}^2-{\omega_2}^2)}
&\\
&+\frac{ ( {a_{{1}}}^{2}a_{{6}}-2a_{{1}}a_{{2}}a_{{5}}+2a_{{1}}a_{{2}}a_{{9}}-{a_{{2}}}^{2}a_{{6}})}{16{\omega_1}({\omega_1}^2-{\omega_2}^2)(a_{{12}}+a_{{14}} ) ^{-1}}\!-\!\frac{a_{{1}}a_{{2}}( {a_{{7}}}^{2}+{a_{{8}}}^{2}\!-\!{a_{{10}}}^{2}\!-\!{a_{{11}}}^{2} ) \!-\!({a_{1}}^2-{a_{2}}^2)(a_{{10}}a_{{7}}+a_{{8}}a_{{11}})}{4{\omega_1}({\omega_1}^2-{\omega_2}^2)}
\!+\!\frac{{\omega_1}^{-1}(a_{{1}}a_{{5}}+3a_{{1}}a_{{9}}\!-\!a_{{2}}a_{{6}})}{4({\omega_1}^2-{\omega_2}^2)(a_{{3}}a_{{10}}+a_{{4}}a_{{11}})^{-1}}&\\
& -\frac{ ( a_{{3}}a_{{7}}+a_{{4}}a_{{8}} )  ( a_{{1}}a_{{6}}-3a_{{2}}a_{{5}}-a_{{2}}a_{{9}} )}{4{\omega_1}({\omega_1}^2-{\omega_2}^2)}-\frac{(a_{{1}}a_{{3}}a_{{8}}-a_{{1}}a_{{4}}a_{{7}}-a_{{2}}a_{{3}}a_{{11}}+a_{{2}}a_{{4}}a_{{10}} ) +a_{{3}}a_{{6}} ( a_{{1}}a_{{11}}+a_{{2}}a_{{8}} ) -a_{{4}}a_{{6}} ( a_{{1}}a_{{10}}+a_{{2}}a_{{7}})}{8{\omega_1}^{2}{\omega_2}^{-1}({\omega_1}^2-{\omega_2}^2)(a_{{5}}-a_{{9}})^{-1}},&
\end{eqnarray*}
\begin{eqnarray*}
&b_{1, 2}:=\frac{3 ({a_{{3}}}^{2}+{a_{{4}}}^{2})^{2}( a_{{9}}+a_{{5}} )}{16{\omega_1}^{-2}{\omega_2}^4({\omega_1}^2-{\omega_2}^2)}-\frac{a_{{13}}({a_{{4}}}^{2}-{a_{{3}}}^{2})+2a_{{4}}a_{{3}}(a_{{12}}-a_{{14}})}{16{\omega_1}^{-2}{\omega_2}^{3}(a_{{9}}+a_{{5}})^{-1}({\omega_1}^2-{\omega_2}^2)}
-\frac{3((a_{{2}}a_{{8}}-a_{{1}}a_{{11}})a_{{3}}+a_4(a_{{1}}a_{{10}}-a_{{2}}a_{{7}}))}{4{\omega_1}^{-1}{\omega_2}^{3}({a_{{3}}}^{2}+{a_{{4}}}^{2})^{-1}({\omega_1}^2-{\omega_2}^2)} &\\
&+\frac{(a_{{12}}-a_{{14}})(a_{{1}}a_{{3}}a_{{10}}-a_{{1}}a_{{4}}a_{{11}}-a_{{2}}a_{{3}}a_{{7}}+a_{{2}}a_{{4}}a_{{8}})
+a_{{3}}a_{{13}}(a_{{1}}a_{{11}}-a_{{2}}a_{{8}})+a_{{4}}a_{{13}}(a_{{1}}a_{{10}} -a_{{2}}a_{{7}})}{8{\omega_2}^{2}{\omega_1}^{-1}({\omega_1}^2-{\omega_2}^2)}-\frac{3{a_{{4}}}^{2} ( 2{a_{{3}}}^{2}-{a_{{4}}}^{2} ) ( a_{{9}}+a_{{5}} )}{16{\omega_2}^{2}({\omega_1}^2-{\omega_2}^2)} &\\
&+\frac{ ( {a_{{1}}}^{2}+{a_{{2}}}^{2} )  ( 3{a_{{3}}}^{2}a_{{12}}+9{a_{{3}}}^{2}a_{{14}}+6a_{{13}}a_{{4}}a_{{3}}-3{a_{{4}}}^{2}a_{{12}}-{a_{{4}}}^{2}a_{{14}} )}{8{\omega_2}^{2}({\omega_1}^2-{\omega_2}^2)}-\frac{(3{a_{{3}}}^{2}a_{{12}}+9{a_{{3}}}^{2}a_{{14}}-6a_{{13}}a_{{4}}a_{{3}}+9{a_{{4}}}^{2}a_{{12}}
+3{a_{{4}}}^{2}a_{{14}})({a_{{1}}}^{2}+{a_{{2}}}^{2} )}{8{\omega_1}^{2}({\omega_1}^2-{\omega_2}^2)}&\\
&-\frac{( {a_{{3}}}^2-{a_{{4}}}^2 )( a_{{7}}a_{{8}}+a_{{10}}a_{{11}} ) + ( -{a_{{7}}}^{2}+{a_{{8}}}^{2}-{a_{{10}}}^{2}+{a_{{11}}}^{2} ) a_{{4}}a_{{3}}-a_{{13}}a_{{3}} ( a_{{1}}a_{{7}}+a_{{2}}a_{{10}} )+ ( 8a_{{3}}a_{{12}}+a_{{13}}a_{{4}} )  ( a_{{1}}a_{{8}}+a_{{2}}a_{{11}}) }{4{\omega_2}({\omega_1}^2-{\omega_2}^2)}&\\
&+\frac{2a_{{13}}( a_{{1}}a_{{3}}a_{{11}}+a_{{1}}a_{{4}}a_{{10}}-a_{{2}}a_{{3}}a_{{8}}-a_{{2}}a_{{4}}a_{{7}})-2a_{{4}} ( 5a_{{12}}+3a_{{14}} )  ( a_{{1}}a_{{11}}-a_{{2}}a_{{8}} )}{16{\omega_1}({\omega_1}^2-{\omega_2}^2)}-\frac{a_{{3}} ( 3a_{{12}}+5a_{{14}})( a_{{1}}a_{{10}}-a_{{2}}a_{{7}})}{8{\omega_1}({\omega_1}^2-{\omega_2}^2)}&\\
&+\frac{3(  ( -a_{{1}}a_{{11}}+a_{{2}}a_{{8}} ) a_{{3}}+a_{{4}} ( a_{{1}}a_{{10}}-a_{{2}}a_{{7}}))
({a_{{3}}}^{2}+{a_{{4}}}^{2})}{4{\omega_1}{\omega_2}({\omega_1}^2-{\omega_2}^2)}.&
\end{eqnarray*}
\end{footnotesize}

\end{document}